\title{\bf{Grauert's Approximation Theorem in any Characteristic and Applications}}
\author{
       \bf{Gert-Martin Greuel and Gerhard Pfister}\\
 }
\date{}

\documentclass[11pt]{article}
\usepackage{amsmath, amsthm, amssymb, mathrsfs, amscd, amsthm, wasysym, amscd,amsfonts,graphicx}
\usepackage{indentfirst,url,xypic, xcolor}
\usepackage{lipsum} 
\usepackage[all]{xy}
\usepackage{url}
\usepackage[colorlinks,plainpages]{hyperref}
\usepackage{verbatim}
\usepackage{fancyvrb}
\usepackage{listings}
\usepackage{eucal}				
\usepackage[T1]{fontenc}
\usepackage[utf8]{inputenc}
\usepackage{indentfirst,url,xypic}
\usepackage{lipsum}
\usepackage{faktor}
\usepackage{url}
\usepackage[colorlinks,plainpages]{hyperref}
\usepackage{verbatim}
\usepackage{fancyvrb}
\usepackage{listings}
\usepackage{mathtools}
\usepackage{algorithm}
\usepackage[noend]{algorithmic}
\usepackage{listings}
\usepackage{enumitem}  		
\usepackage{bm} 			
\usepackage{blindtext}
\usepackage{multicol}
\usepackage{}				

\usepackage[colorlinks,plainpages]{hyperref}
\hypersetup{
	colorlinks=true,
	linkcolor=blue,
	filecolor=magenta,      
	urlcolor=cyan,
	citecolor=blue
}

\DeclareMathOperator{\rank}{rank}

\DeclareMathOperator{\id}{id}

\DeclareMathOperator{\modular}{mod}

\DeclareMathOperator{\ord}{ord}
\DeclareMathOperator{\order}{ord}

\DeclareMathOperator{\rsim}{\stackrel{r}{\sim}}
\DeclareMathOperator{\csim}{\stackrel{c}{\sim}}
\DeclareMathOperator{\Id}{{\mathbf 1}}
\DeclareMathOperator{\LM}{LM}
\DeclareMathOperator{\LC}{LC}
\DeclareMathOperator{\LT}{LT}
\DeclareMathOperator{\tail}{tail}
\DeclareMathOperator{\NF}{NF}
\DeclareMathOperator{\Spec}{Spec}

\DeclareMathOperator{\Sing}{Sing}

\newtheorem{Definition}{Definition}[section]
\newtheorem{Theorem}[Definition]{Theorem}
\newtheorem{Remark}[Definition]{Remark}
\newtheorem{Proposition}[Definition]{Proposition }
\newtheorem{Corollary}[Definition]{Corollary}

\newtheorem{Lemma}[Definition]{Lemma}

\newcommand{\N}{{\mathbb N}}

\newcommand{\Z}{{\mathbb Z}}
\newcommand{\R}{{\mathbb R}}
\newcommand{\C}{{\mathbb C}}
\newcommand{\Q}{{\mathbb Q}}

\newcommand{\kq}{{\mathcal Q}}
\newcommand{\kr}{{\mathcal R}}

\newcommand{\fm}{\mathfrak{m}}

\newcommand{\bx}{\boldsymbol{x}}

\newcommand{\beps}{\boldsymbol{\varepsilon}}
\newcommand{\eps}{\varepsilon}

\def\<bx{\langle \bx \rangle}

\begin{document}
\maketitle
\section*{Abstract}
In his seminal Inventiones paper from 1972 Grauert proved the existence of a semiuniversal deformation of an arbitrary complex analytic isolated singularity. 
For the proof he invented an approximation theorem for solving a system of "nested" analytic equations, which is now called Grauert's approximation theorem.
To prove this, Grauert introduced standard bases for ideals in power series rings and proved a generalized Weiertra{\ss}  division theorem.  All this was done for convergent power series over the complex numbers.

The purpose of this article is to extend Grauert's division and approximation theorem to convergent power series over arbitrary real valued fields of any characteristic. As an application, which was actually the motivation for this article, we derive  the existence of a semiuniversal deformation for an isolated singularity and a splitting lemma for not necessarily isolated hypersurface singularities over a real valued field in any characteristic. 

\tableofcontents

\section{Introduction} \label{sec:1}
Grauert's paper  \cite{Gr}, "\" Uber die Deformation isolierter Singularit\"aten analytischer Mengen" was a landmark in singularity theory. Not only because of the proof of the existence of a semiuniversal deformation of an isolated singularity but also because of the methods.  Grauert introduced what is now called 'standard bases' for ideals in power series rings and he proved a generalized Weiertra{\ss}  division theorem (division by an ideal). He used this to prove an approximation theorem for solving a system of ”nested” analytic equations, the fundamental result to prove the existence of a semiuniversal deformation.
Semiuniversal deformations of an isolated singularity are a cornerstone in singularity theory and have been studied in numerous cases. Standard bases for power series ideals\,\footnote{Standard bases were independently discovered by Hironaka in his work on the resolution of singularities.} are a local analog of Gr\"obner bases for polynomial ideals, both of which are now indispensable tools for the theoretical and computational study of systems of equations in algebraic and analytic local rings resp. in polynomial rings.  In this article we
extend Grauert's results to power series over a real valued field of any characteristic.
Note that power series over a real valued field include {\em convergent power series over $\R$ and $\C$ and over many other fields} of interest (see Remark \ref{rm.rvf}), as well as {\em formal power series over arbitrary fields}.

The motivation for extending Grauert's division and approximation theorem as well as the existence theorem for a semiuniversal deformation resulted from our work on the classification of singularities in positive characteristic. For this we needed a splitting lemma (originally due to Thom for differentiable functions) that splits a power series of order two in its quadratic part and a part of order three in independent variables, which is  indispensable for the classification of singularities in more than two variables. In \cite{GP25} we proved the splitting lemma, also for non-isolated singularities, for formal  and for convergent real or complex power series.  
The case of convergent power series in positive characteristic was left open, because we needed for the proof the semiuniversal unfolding of a hypersurface, which was only available  for real or complex analytic singularities. 
In this article we provide a proof of Grauert's  approximation theorem over a real valued field and derive from that the existence of a convergent semiuniversal deformation for every isolated singularity.  As a consequence we prove a  general splitting lemma that encompasses all known cases and goes beyond them.

 Grauert invented his approximation theorem, since Artin's "nested" approximation theorem does not hold in the {\em analytic} case.
 There are various Artin approximation theorems, and they are a keystone in algebraic geometry and singularity theory. Here we mention only two of them: the analytic approximation theorem and the 
 nested {\em algebraic} approximation theorem.
 For a detailed discussion on Artin approximation theorems see, for example, \cite{Ha17} and \cite{Ro18}.
 
The methods used to prove Grauert's division and approximation theorems do not differ substantially from the classical methods used in the complex analytic case; indeed, parts of Sections 2, 3, and 5 follow the arguments presented in \cite{GG} and \cite{dJP00}. However, it is not immediately obvious that these methods remain valid in arbitrary characteristic over real valued fields. Establishing this carefully—thereby providing a reliable reference—is one of the primary purposes of our work. For the sake of completeness and clarity, we include the necessary auxiliary lemmas, occasionally referring to \cite{GG} and \cite{dJP00} for proofs where appropriate. In contrast, the applications presented in Section 4 are novel; they build upon the results of Sections 2 and 3 and constitute the primary motivation for our work.\\

Let us first fix some notations. Let $K$ be a field of arbitrary characteristic with a real valuation $| \ | :K \to \R_{\ge 0}$\,\footnote{This means: $ |a| = 0 \Leftrightarrow a=0$, $ |a||b| = |ab|$, and $|a+b| \le |a|+|b|$.}. Let $x=x_1,\ldots,x_n$.
For each \mbox{${\eps}\in
  (\R_{>0})^n$}, we
define a map 
$\|\phantom{f}\|_{\eps}:K[[x]]\to \R_{>0}\cup  \{\infty\}$
   by setting for a formal power series
   \mbox{$f  =  \sum_{{\alpha}\in\N^n}c_{\alpha} {x}^{\alpha} $}, $c_\alpha \in K$, 
   $\alpha=(\alpha_1,...,\alpha_n)$, 
   $x^\alpha=x_1^{\alpha_1}\cdot...\cdot x_n^{\alpha_n}$,
$$\|f\|_{\eps}:= \sum_{{\alpha} \in \N^n} |c_{\alpha}|\cdot
{\eps}^{{\alpha}}\in \R_{\ge 0}\cup \{\infty\}\,.$$
The power series
\mbox{$f  =  \sum_{{\alpha}\in\N^n}c_{\alpha} {x}^{\alpha} $}
is called {\em
convergent}
iff there exists a real vector \mbox{$\eps\in(\R_{>0})^n$}
such that \mbox{$ \| f \|_{\eps} < \infty $}.
We denote by 
$$K\{x\}=K\{x_1,...,x_n\}$$
the {\em convergent (or analytic) power series ring} over $K$. The analytic $K$-algebra 
$K\{ x\}$ is a Noetherian, integral and factorial local ring  with maximal ideal $\fm =  \langle x_1,...,x_n \rangle$, and morphisms being morphisms of local $K$-algebras (see \cite[Section I.1.2]{GLS25}). 
We denote  the {\em order} of $f \in K\{x\}$ 
by 
$$\ord(f) := \max \{l \mid f \in \fm^l\}\text{ if }f \ne 0\text{ and }\ord(0)=\infty.$$

 If  $K$  is complete w.r.t. $| \ |$ (every Cauchy sequence with respect to $| \ |$ converges in $K$) then $K$ is called a {\em complete real valued field}.
Note that $\bar K\{x\}\cap K[[x]]=K\{x\}$, where $\bar K$ denotes the completion of $K$. See Remark \ref{rm.rvf} for examples of real valued fields.

For a geometric interpretation of the results of this paper we mention, that if the valuation $| \ |$ on $K$ is  complete, then any convergent power series 
 $f= \sum c_\alpha x^\alpha \in K\{x\}$, $\|f\|_{\beps}< \infty$, defines an analytic function $P^n_\eps \to K, \ a \mapsto \sum c_\alpha a^\alpha$, with $P^n_\eps$ the policylinder
 $$P^n_\eps: =
\{a=(a_1,...,a_n) \in K^n \mid |a_i| < \eps_i, i =1,...,n\}$$
(see \cite[§1.1, Bemerkung]{GR71}). If $K$ is not complete we can  pass to the algebraic closure $\tilde K$ of the completion $\bar K$ and all statements of this article could be formulated geometrically over $\tilde K$ as in the case $K=\C$ and $| \ |$ the absolute value.\\

We extend $\| \ \|_{\eps}$ to $f=(f_1,\ldots,f_N)\in K[[x ]]^N=
\sum_{i=1}^N K[[x]]e_i$
by 
$$\|f\|_{\eps}:= \sum_{i=1}^N\|f_i\|_{\eps}.$$
Let $B_{\eps}^N=\{f\in K[[x]]^N \text{  } | \text{ } \|f\|_{\eps}<\infty\}$. Then $K\{x\}^N=\cup_{\eps\in\R^n_{>0}} B_{\eps}^N$ and $\| \ \|_{\eps}$ is a norm on $K\{x\}^N$,  satisfying $ \|fg\|_{\eps}\le  \|f\|_{\eps} \|g\|_{\eps}$.
 If $\rho=(\rho_1,\ldots,\rho_n)\leq\eps=(\eps_1,\ldots,\eps_n)$ then\,\footnote{$\rho\leq\eps$ means $\rho_i\leq\eps_i$ for all $i$.} $B_\eps^N\subset B_\rho^N$.\\
  
  In 1968, M. Artin proved that any formal power series solution of a system of {\em analytic} equations can be approximated by convergent power series solutions.

\begin{Theorem}[Analytic Artin Approximation]
\label{thm.approx}
Let $K$ be a real valued field. If the characteristic of $K$ is positive, we assume additionally that $K$ is quasi-complete\,\footnote{$K$ is called  quasi-complete if the completion $\bar K$  of $K$ is a separable field extension of $K$. Note that in characteristic 0 every real valued field is already quasi-complete.
}. 
Let $x=(x_1,\ldots,x_n)$, $y=(y_1,\ldots,y_m)$ and $f=(f_1,\ldots,f_k)\in K\{ x,y\}^k$. 
Let $\bar y(x)\in K[[x]]^m$
be a formal solution of $f=0$,
$$f(x,\bar y(x))=0.$$
Then there exists for any integer $c>0$ a convergent solution $y_c(x)\in K\{ x\}^m$,
$$f(x,y_c(x))=0,$$
such that $$\bar y(x)\equiv y_c(x) \text{ mod } \langle x\rangle^c.$$
\end{Theorem} 
The following nested approximation theorem strengthens this when the equations are {\em algebraic power series}\,\footnote{An algebraic power series is a formal power series $f(x), x = x_1,...,x_n,$ for which a non-zero polynomial $P(x, t)$ exists, such that $P(x, f(x)) = 0$ holds. For any field $K$  the ring of algebraic power series is denoted by $K\langle x \rangle$.}, which involve variables that are “nested,” i.e., each block of unknowns depends only on a subset of parameters.
Then for every $c$, there exists an algebraic solution with the same nesting condition, which agrees with the formal one modulo $\langle x\rangle^c$.
\begin{Theorem}[Nested Algebraic Artin Approximation]\label{thm.nestapprox}
Let $K$ be any field. 
Let $x=(x_1,\ldots,x_n)$ and $y=(y_1,\ldots,y_m)$ and $f=(f_1,\ldots,f_k)\in K\langle x,y\rangle^k$ be a system of algebraic power series. 
Let $\bar y(x)=(\bar y_1(x),\ldots,\bar y_m(x)),$ with $ \bar y_i(x) \in K[[x_1,\ldots,x_{s_i}]]$ for some $s_i\leq n$, $i=1,...,m$,
be a formal solution of $f=0$,
$$f(x,\bar y(x))=0.$$
Then there exists for any integer $c>0$ an algebraic solution $y_c(x)=(y_{c,1}(x),\ldots,y_{c,n}(x)),$  with $y_{c,i}\in K\langle x_1,\ldots,x_{s_i}\rangle$,
$$f(x,y_c(x))=0,$$
such that $$\bar y(x)\equiv y_c(x) \text{ mod } \langle x\rangle^c.$$
\end{Theorem}

\begin{Remark}\label{rm.nestapprox}
{\em
\begin{enumerate}
\item[(i)]The analytic approximation theorem was proved in characteristic $0$ by M. Artin \cite{Ar68}. In characteristic $p>0$ for complete valued fields it was proved 
by M. Andr\'e \cite{An75} and for quasi-complete fields by K.-P. Schemmel \cite{Sc82}.\footnote{Schemmel proved that the property of $K$ being quasi-complete is necessary and sufficient for the Approximation Theorem to be true.}
\item[(ii)]The nested algebraic approximation theorem was proved by D. Popescu \cite{Po86} solving a problem of M. Artin posed in \cite{Ar69}. 
That the variables are "nested" becomes clear if we sort the $s_i$ such that $s_1\le s_2\le  ...$. Then $\bar y_1$ depends only on $x_1,...,x_{s_1}$,  $\bar y_2$ only on $x_1,...,x_{s_2}$ etc., and the same holds for $y_{c,1} $, $y_{c,2}$, etc..
\item[(iii)] A similar nested theorem in the analytic case is wrong, see \cite{Ga71}. 
H. Grauert  proved in \cite{Gr} under additional assumptions
 a nested complex analytic approximation theorem, that was used to construct an analytic semiuniversal deformation of a complex analytic germ with an isolated singularity. In his proof Grauert developed (in today's language) a theory of standard bases  using his division theorem to get a normal form of a function in the complex analytic case.\\
In this paper we generalize Grauert's division and approximation theorem (Theorem \ref{thm.div} and Theorem \ref{grauertapprox}) to any real valued field and deduce the existence of a semiuniversal deformation (Theorem \ref{thm.sudef}) for convergent power series over such fields.
\item[(iv)] The statements in Theorem \ref{thm.approx} (resp. Theorem \ref{thm.nestapprox}) can be generalized to solutions modulo an ideal.\\
Let $I=\langle h_1,\ldots,h_l\rangle$ be an ideal in $K\{x\}$ (resp. in $K\langle x\rangle$). If we replace 
in Theorem \ref{thm.approx} (resp. Theorem \ref{thm.nestapprox}) the condition
$$f(x,\bar y(x))=0 \ \text{ by } \ f(x,\bar y(x))\equiv 0 \mod IK[[x]]$$ 
we can find an analytic (resp. nested algebraic) solution $\mod IK\{x\}$ (resp. $\mod IK\langle x\rangle$).\\
This is in fact equivalent to the statements in Theorems \ref{thm.approx} and \ref{thm.nestapprox}:  Replace the $f_j$ by $f_j+\sum_{i=1}^lz_{ij}h_j$, $(z_{ij})$ new variables, and then apply the theorems to the new system.
 \end{enumerate}}
\end{Remark}

\begin{Remark}\label{rm.rvf}
{\em
Some examples of real valued fields.
\begin{enumerate}
\item [(1)] We allow the trivial valuation  ($|0| =0, |a| =1$ if $a\ne 0$) and then $K\{x\}=K[[x]]$, but the results of this article are interesting (and new) only if the valuation is non-trivial. Finite fields permit only the trivial valuation. The valuation may be non-Archimedian, satisfying the stronger triangle inequality  $|x+y|\le \max(|x|,|y|)$.
\item  [(2)] Examples of a complete real valued field are $\C$  
resp. $\R$ with the usual absolute value as valuation. The absolute value on the field of rational numbers makes $\Q$ a real valued field, which is not complete (but it is quasi-complete since its characteristic is 0).
\item  [(3)] Another example is given by the field of $p$-adic numbers.
Let \mbox{$p$} be a prime number, then the map
$ v:\Z\setminus\{0\} \to \R_{>0}\,,\quad a \mapsto p^{-m} \text{ with }
m:=\max\{k\in \N  \mid p^k \text{ divides }a\} $
extends to a unique real valuation of $\Q$. With this
valuation, $\Q$ is a real valued field that is not complete but quasi-complete.
 The completion $\bar \Q$ with resp. to  $v$ is called the
field of {\em $p$-adic numbers}.
\item  [(4)]  Let $F$ be a field and $K$ the quotient field of $F[x]$. If $ f/g \in K\smallsetminus 0$ then
$\ord(f/g) = \ord(f)-\ord(g) \in \Z$. We set
$|f/g| = e^{-\ord(f/g)}$ and $|0|=0$. Then $| \ |$ is a real valuation on $K$. The same applies to $K$ the quotient field of $F\{x\}$ (with $F$ real valued). 
\item  [(5)] The field of {\em formal Laurent series} $K=F((T))$ is the quotient field of $F[[T]]$. With valuation $|c| = e^{-m}$ if $c=\sum_{i=m}^\infty a_iT^i, m \in \Z, a_m\ne0$, $K$ is a real valued field. Another examples of a real valued field  is the field of {\em Puiseux series}  $K=\cup_{n=1}^\infty F((T^{\frac{1}{n}}))$,  with valuation $|c| = e^{-m/n}$ 
if $c=\sum_{i=m}^\infty a_iT^{\frac{i}{n}}, a_m\ne0$.
\end{enumerate}}
\end{Remark}

\section{Grauert's Division Theorem} \label{sec:2}
In this section we generalize Grauert's Division theorem, as well as standard bases, to convergent power series over real valued fields.

We use the notations of the theory of standard bases,  as developed in \cite[Section 1.6, 2.3, 6.4]{GG}. For example, for a fixed monomial ordering $>$
on (the set of monomials $\{x^{\alpha}\}$ of) $K[[x]]$, we write 
$f \in K[[x]], f \ne 0 $, in a unique way as sum of non-zero terms
$$f  = c_{\alpha} {x}^{\alpha} +  c_{\beta} {x}^{\beta} + c_{\gamma} {x}^{\gamma} + ... ,$$
 with ${x}^{\alpha} > {x}^{\beta} >  {x}^{\gamma} > ... ,$
and $c_{\alpha},  c_{\beta}, c_{\gamma}, ...  \in K$.
$\LM (f) := x^{\alpha}$
 is called the {\em leading monomial}, $\LC (f):= c_{\alpha}$ the {\em leading coefficient},  $\LT (f):= c_{\alpha}{x}^{\alpha}$ the {\em leading term}, and $\tail (f):= f - \LT (f)$ the {\em tail} of $f$. 
 
From now on we fix the {\em local degree ordering} deglex
 $>$ on  $K[[x]]$. This means that for $\alpha=(\alpha_1,\ldots,\alpha_n)$ and $\beta=(\beta_1,\ldots,\beta_n)$ we have (with  $|\alpha | =\alpha_1+...+\alpha_n$)
 $$
 \begin{array}{lll}
 x^\alpha > x^\beta \text{ iff } & |\alpha | <|\beta  \text{ or }\\
 &|\alpha |=|\beta | \text{ and } \alpha_1=\beta_1,\ldots,\alpha_{l-1}=\beta_{l-1} \text{ and } \alpha_l < \beta_l \text{ for some } l.
 \end{array}
 $$
We extend $>$ to (the set of {\em monomials} $\{x^{\alpha}e_i\}$ of)
$K[[x]]^N$ by setting
$$x^{\alpha}e_i > x^{\beta} e_j \text{ iff } x^{\alpha} > x^{\beta} \text{ or } (  x^{\alpha} = x^{\beta} \text{ and } i<j),$$
where $\{e_1,...,e_N\}$ is the canonical basis of $K[[x]]^N$.
Any vector
$f \in K[[x]]^N \smallsetminus \{0\}$ can be written uniquely as
$$f = c_\alpha x^\alpha e_i + f^*$$
with $c_\alpha \in K \smallsetminus \{0\}$ and $x^\alpha e_i > x^{\alpha^*} e_j$ for any non-zero term $c^*x^{\alpha^*} e_j$ of $f^*$. We define
$$LM(f) := x^\alpha e_i, \
LC(f):= c_\alpha, \ LT(f) := c_\alpha  x^\alpha e_i \
\tail(f) := f - LT(f)$$
and call it the {\em leading monomial, leading coefficient, leading term} and {\em  tail}, respectively, of $f$. 
Note  that for a sequence $w_j \in K[[x]]^N$ with $LM(w_j) > LM(w_{j+1})$ we have $\sum_j w_j \in  K[[x]]^N$.\\

The following two lemmas are needed for the proof of Grauert's division theorem. The proofs are postponed to the end of this section. $K$ denotes an arbitrary real valued field.

\begin{Lemma}\label{L1}
If $K$ is a complete real valued field then $B_{\eps}^N$ is complete\,\footnote{$K$ resp. $B_{\eps}^N$ is complete, iff every  Cauchy sequence w.r.t. $| \  |$ in $K$ resp. w.r.t. $\| \ \|$ in $B_{\eps}^N$ converges.} for $\eps >0$.
\end{Lemma}

\begin{Lemma}\label{L2}
Let $f_1,\ldots,f_m\in K\{x\}^N$.
Given $\eps >0$, there exists a $\delta\in (\R_{>0})^n$ such that $\|{\tail}(f_i)\|_{\delta} < \eps\|\LM(f_i)\|_{\delta}$ for all $i$.
\end{Lemma}

The set $M=\{x^{\alpha}e_i \text{ } | \text{ } \alpha\in \N^n , i\in \{1,\ldots,N\}\}$ of {\em monomials of $K[[x]]^N$}
can be identified with  $\N^n\times \{1,\ldots,N\}$. This set is a $\N^n$ semi-module by the canonical action
$\alpha + (\beta,k)=(\alpha+\beta,k)$.

The following theorem (”division by an ideal", a generalization of the Weierstra$\ss$  division theorem) is proved for $K=\C$ by Grauert \cite{Gr}, see also \cite{dJP00}. The formal case is proved in \cite{GG}.
\begin{Theorem}[Division Theorem for real valued fields]\label{thm.div}
Let  $f_1, \dots, f_m$ be elements of $K\{x\}^N$. Then, 
  for every $f \in K\{x\}^N$, there exist
$q_1, \ldots, q_m \in K\{x\}$ and an element $r \in K\{x\}^N$
such that
\[
f = \sum^m_{j=1} q_jf_j + r,
\]
having the following two  properties:
\begin{enumerate}
\item No monomial of $r$ is divisible by $\LM(f_j)$ for $j = 1, \dots, m$; $r$ is uniquely determined with this property  by  the ordered set  $f_1, \dots, f_m$;
\item $\LM(f) \ge \LM(q_j f_j)$ for  $j = 1, \dots, m$.
\end{enumerate}
\end{Theorem}

\begin{proof}
We first give a construction of $q_i\in K[[x]]$ and $r\in K[[x]]^N$ having the properties (1) and (2). Then we pass to the completion $\bar K$ of $K$ to prove the convergence. Using $\bar K\{x\}\cap K[[x]]=K\{x\}$ we obtain the convergence
over $K$.\\

We may assume that $\LC(f_i)=1$ for all $i$.  Let
$\LM(f_i)=x^{\alpha(i)}e_{k(i)}$. We define the $N^n$ semi-module 
$$\Gamma:=\langle (\alpha(1),k(1)),\ldots,(\alpha(m),k(m))\rangle_{\N^n}$$
and a partition $\Gamma_1,...,\Gamma_m$ of $\Gamma$ by $\Gamma_1:=\langle (\alpha(1),k(1))\rangle_{\N^n}$ and for $i=2,\ldots m$
$$\Gamma_i:=\langle (\alpha(i),k(i))\rangle_{\N^n}\cap (\Gamma\setminus \langle(\alpha(1),k(1)),\ldots,(\alpha(i-1),k(i-1))\rangle_{\N^n}).$$
Now we define a map $r:K[[x]]^N\longrightarrow K[[x]]^N$ and maps $q_i:K[[x]]^N\longrightarrow K[[x]]$ as follows:
Given $w=\sum_{\alpha,k}w_{\alpha,k}x^{\alpha}e_k\in K[[x]]^N$ we set
$$r(w):=\sum_{(\alpha,k)\notin \Gamma}w_{\alpha,k}x^{\alpha}e_k \text{ and } q_i(w):=\frac{1}{x^{\alpha(i)}}\sum_{(\alpha,k(i))\in\Gamma_i}w_{\alpha,k(i)}x^{\alpha}.$$
Note that if $(\alpha,k(i))\in \Gamma_i$ and $w_{\alpha,k(i)}\ne 0$, then $x^{\alpha(i)} |w_{\alpha,k(i)} x^\alpha$. Hence
$r(K\{x\}^N)\subset K\{x\}^N$ and $q_i(K\{x\}^N)\subset K\{x\}$.
We obtain 
\begin{equation}\label{xx}
w=\sum_{i=1}^mq_i(w)x^{\alpha(i)}e_{k(i)}+r(w)=\sum_{i=1}^mq_i(w)\LM(f_i)+r(w).\hspace{2.5cm} 
\end{equation}
We define now a sequence $\{w_j\}_{j\ge 0}$ inductively by
\begin{equation}\label{yy}
w_0:=f \text{ and } w_{j+1}:=w_j-\sum_{i=1}^m q_i(w_j)f_i-r(w_j)=-\sum_{i=1}^mq_i(w_j)\text{tail}(f_i). 
\end{equation}
Since the ordering is a local degree ordering and 
$\LM(w_j)> \LM(w_{j+1})$ 
the element $w:=\sum_{j=0}^\infty w_j$ is a well-defined element in $K[[x]]^N$.  Now using (\ref{yy})  we obtain
$$f=\sum_{j=0}^\infty(w_j-w_{j+1})=\sum_{j=0}^\infty(w_j+\sum_{i=0}^mq_i(w_j)\text{tail}(f_i))=\sum_{j=0}^\infty(\sum_{i=0}^mq_i(w_j)f_i+r(w_j)).$$
Since $\sum_{j=0}^\infty q_l(w_j)=q_l(w)$ and  $\sum_{j=0}^\infty r(w_j)=r(w)$ we obtain
$$f=\sum_{i=0}^m q_i(w)f_i+r(w).$$
By definition of $r$ we have that no monomial of $r(w)$ is divisible by $\LM(f_j)$ for $j = 1, \dots, m$. $r$ is uniquely determined by  construction. This proves statement 1. \\

\noindent To see 2. note that $\LM(w_j)>\LM(w_{j+1})$. This implies
$$\LM(f)\geq \LM(q_i(f)f_i)=\LM(q_i(w_0)f_i)\geq\LM(\sum_{j=0}^\infty q_i(w_j)f_i)=\LM(q_i(w)f_i),$$
which proves the formal case.\\

It remains to prove that $w =\sum_{j=0}^\infty w_j \in K\{x\}^N$, which implies $r(w)\in K\{x\}^N$ and $q_i(w)\in K\{x\}$ for all $i$. For this part of the proof we may assume that
$K$ is a complete valued field (see beginning of the proof). Let $\eps>0$ be given. Using Lemma \ref{L2} we find $\delta\in R^n_{>0}$ such that for all $i$
$$\|\text{tail}(f_i)\|_\delta < \eps\|\LM(f_i)\|_\delta.$$
Now by definition of $r$ and $q_i$ we have
\begin{equation}\label{zz}
\|r(w)\|_\delta\leq \|w\|_\delta \text{ and }  \|q_i(w)\|_\delta \leq\frac{\|w\|_\delta}{\delta^{\alpha(i)}}=\frac{\|w\|_\delta}{\|\LM(f_i)\|_\delta}
\end{equation}
We obtain
$$\|w_{j+1}\|_\delta=\|-\sum_{i=1}^mq_i(w_j)\text{tail}(f_i)\|_\delta\leq\sum_{i=1}^m\frac{1}{\delta^{\alpha(i)}}\|w_j\|_\delta\frac{\eps}{m}\delta^{\alpha(i)}=\eps\|w_j\|_\delta.$$
We obtain $\|w_{j+1}\|_\delta\leq\eps^{j+1}\|f\|_\delta$. This implies that $\sum_{j=0}^k w_j$ is a Cauchy sequence in
$B_\delta^N$. Since $B_\delta^N$ is complete (Lemma \ref{L1}), it follows that $w=\sum_{j=0}^\infty w_j\in B_\delta^N$.
\end{proof}

We note that the proof Theorem \ref{thm.div} works not only for 
deglex, but for any local degree ordering $>$ on $K[[x]]$ (see \cite{GG}), and any extension on $K[[x]]^N$.

\begin{Remark}\label{normNF} {\em The inequality 
 $\|w_{j+1}\|_\delta\leq\eps^{j+1}\|f\|_\delta$ implies $$\|w\|_\delta\leq\sum_{j=0}^\infty\|w_j\|_\delta\leq \|f\|_\delta\sum_{j=0}^\infty\eps^j=\frac{1}{1-\eps}\|f\|_\delta.$$
Equation {\em (\ref{zz})} implies 
 $$\|r(w)\|_\delta\leq \frac{1}{1-\eps}\|f\|_\delta  \text{ and } \|q_j(w)\|_\delta\leq \frac{1}{1-\eps}\frac{\|f\|_\delta}{\|\LM(f_j)\|_\delta}.$$
 Moreover, if $f_1,\ldots,f_m\in K\{x_{s+1},\ldots,x_n\}$ for some $s\geq 0$, then the proof of Theorem \ref{thm.div} implies
  $$\|r(w)\|_{\lambda\circ\delta}\leq \frac{1}{1-\eps}\|f\|_{\lambda\circ\delta}  \text{ and } \|q_j(w)\|_{\lambda\circ\delta}\leq \frac{1}{1-\eps}\frac{\|f\|_{\lambda\circ\delta}}{\|\LM(f_j)\|_{\lambda\circ\delta}}$$
 with $0<\lambda<1$ and $\lambda\circ\delta=(\delta_1,\ldots,\delta_s,\lambda\delta_{s+1},\ldots,\lambda\delta_n)$.
 }
\end{Remark}

We define now a standard basis for a fixed monomial ordering $>$ on $K[[x]]^N$.

\begin{Definition}\label{def.SB} 
Let $S$ be a subset of $K\{x\}^N$ and $I \subset K\{x\}^N$ a submodule.
\begin{enumerate}
\item The $K[x]$-submodule of  $K[x]^N$
$$L(S) := \langle LM(g) \mid g \in S \smallsetminus \{0\}\rangle _{K[x]}\subset  K[x]^N$$
is called the {\em leading module} of $\langle S \rangle_{K\{x\}}$. In particular,
$L(I)$ is called the {\em leading module} of $I$.
\item A finite set $S \subset K\{x\}^N$ is called a {\em standard basis} of $I$ if 
\begin{enumerate}
\item $S\subset I$;
\item $L(S) = L(I)$, that is, for any $f \in I \smallsetminus \{0\}$ there exists
a $g \in S$ satisfying $LM(g) \mid LM(f)$.\footnote{We say that $x^\beta e_i$ is divisible by $x^\alpha e_j$ (notation $x^\alpha e_i \mid x^\beta e_j$) if $i = j$ and $x^\alpha \mid x^\beta$.}
\end{enumerate}
\end{enumerate}
\end{Definition}

\begin{Definition}\label{def.NF} 
With the notations of Theorem \ref{thm.div} we define the {\em normal form} of $f$ with respect to the ordered set
 $S:=\{f_1, \dots, f_m\}$ by 
 $$\NF(f | S):=r.$$
 Let $I\subset K\{x\}^N$ be a submodule and $S:=\{f_1, \dots, f_m\}$  a standard basis of $I$. 
 We define the {\em normal form of $f$ with respect to $I$} as
 $$\NF(f | I):=\NF(f | S).$$ 
\end{Definition}
By Lemma \ref{lem.NF} below,  $\NF(f | S)$ does not depend on the ordering of $S$ and also not on the choice of a standard basis, but only on the ideal $I$ generated by $S$.  Hence, the notation $\NF(f | I)$ makes sense.\\

We recall the basic properties of standard bases and normal forms for submodules of $K\{x\}^N$. We fix a local degree ordering $>$  on $K[[x]]$, for example deglex, and an extension to  $K[[x]]^N$.\\

\begin{Lemma}\label{lem.NF}
Let $I\subset K\{x\}^N$ be a submodule and $S:=\{f_1, \dots, f_m\}$ a standard basis of $I$.
\begin{enumerate}
\item[(1)] For any $f\in k\{x\}^N$ we have $f \in I$ if and only if $NF(f | S) = 0$.
\item[(2)] If $J\subset K\{x\}^N$ is a submodule with $I \subset J$, then $L(I) = L(J)$ implies $I= J$.
\item[(3)] $I = \langle S \rangle_{K\{x\}}$, that is, the standard basis $S$ generates the submodule $I$.
\item[(4)] $NF(f | S)$ does not depend on the ordering of $S$ and also not on the choice of a standard basis of $I$.
\end{enumerate}
\end{Lemma}
\begin{proof}
(1) If $NF(f | S) = 0$ then $f \in I$ by Theorem \ref{thm.div}.
 If $NF(f | S) \ne 0$,
then $LM(NF(f | S)) \notin L(S) = L(I)$ again by Theorem \ref{thm.div}. Hence $NF(f | S) \notin I$, which implies $f \notin I$, since $\langle S\rangle \subset I$. \\
To prove (2), let $f \in J$ and assume that $NF(f|S) \ne 0$.
Then $LM(NF(f | S)) \notin L(S) = L(I) = L(J)$, contradicting (by Theorem \ref{thm.div})
$NF(f | S) \in J$.
Hence, $f \in I$ by (1).\\
(3) follows from (2), since $L(I) = L(S) \subset L((S)_{K\{x\}}) \subset L(I)$. \\
Finally, to prove (4), let $f \in K\{x\}^N$ and assume
that $r$ resp.  $r'$ are two normal forms of $f$ with respect to the standard bases $S$ resp. $S'$ of $I$. Then no monomial of $r$ resp. $r'$ is divisible by any monomial of $L(S)$ resp. $L(S')$. Moreover, 
$r - r' = (f - r') - (f - r) \in \langle S' \rangle_{K\{x\}} +\langle S \rangle_{K\{x\}} =I$ (by (3)). If 
$r - r' \ne 0$, then $LM(r - r') \in L(I) = L(S)=L(S')$, a contradiction (by Theorem 2.3), since $LM(r - r')$ is a monomial of either $r$ or $r'$.
\end{proof}

\begin{Remark}{\em
The theory of standard bases for ideals ($N=1$) and submodules $I$ of $(K[x]_{\langle x\rangle})^N$ w.r.t. arbitrary local orderings was developed in \cite{GG} with emphasis on algorithms to compute them. Although we do not need it, we like to mention that the following holds (the proof in \cite[Theorem 6.4.3]{GG} for ideals is basically the same for modules):}
\end{Remark}

\begin{Theorem}
If $S \subset K[x]^N$ is standard basis for a local degree ordering of the submodule $I\subset K[x]^N$, then $S$ is a standard basis $I K\{x\}^N$.
\end{Theorem}

This allows to compute a standard basis for submodules $I$ of $K\{x\}^N$ if the generators of $I$ have polynomial components.\\

We prove now Lemma \ref{L1} and Lemma \ref{L2}, which were used in the proof of Theorem \ref{thm.div}.

\begin{proof}[Proof of Lemma \ref{L1}] (See also  \cite[Satz 1]{Gr} for $N=1$.)
Let  $\{f_e = \sum_{\alpha,k} f_{e,\alpha,k}x^\alpha e_k\}_{e\in \N}$ be a Cauchy sequence in $B^N_\rho$.
By definition of the norm we get
 $\rho^\alpha|f_{e,\alpha,k}-f_{m,\alpha,k}| \leq \|f_e -f_m\|_\rho$.
It follows that  for {\em fixed} $(\alpha,k)$ the sequence
$\{f_{e,\alpha,k}\}_{e\in\N}$
is a Cauchy sequence in $K$.
 As
$K$ is complete we can define 
\[
f^{(\alpha,k)} := \lim_{e \to \infty}f_{e, \alpha,k}; \;  \; f: = \sum_{\alpha,k}
f^{(\alpha,k)} x^\alpha e_k \in K[[x]]^N.
\]
We claim that  $f \in B_\rho^N$ and $\lim_{e \to \infty} f_e = f$.
Let $\eps > 0$ be given. Since $f_1, \, f_2, \ldots $ is a Cauchy sequence there exists an $l_0 = l_0(\eps)$ such
that for all $l \ge l_0 $ and all $ e \ge 0$ we have the inequality
\begin{equation}\label{a}
\sum_{\alpha,k} |f_{l+e,\alpha,k} -f_{l,\alpha,k}|\rho^\alpha = \|f_{l+e}-f_l\|_\rho \leq 
\frac{\eps}{2}.
\end{equation}
Furthermore, as $\lim_{e \to \infty}f_{e,\alpha,k} = f^{(\alpha,k)}$
we have that for all $s \in \N$ there exists an $e_0 = e_0(s)$ such that
\begin{equation}\label{b}
\sum_{\alpha: |\alpha|\leq s} \left|f^{(\alpha,k)} -f_{e+l,\alpha,k}\right|\rho^\alpha 
< \frac{\eps}{2N} \quad {\rm for \; all}\; l \ge l_0\; {\rm and}\; e \ge e_0.
\end{equation}
Putting (\ref{a}) and  (\ref{b}) together we get
for all $l \ge l_0$
\[
\begin{array}{lll}
\sum_{\alpha,k: |\alpha|\leq s} \left|f^{(\alpha,k)} -f_{l,\alpha,k}\right|\rho^\alpha 
&\leq 
\sum_{\alpha,k: |\alpha|\leq s} \left|f^{(\alpha,k)} -f_{e+l,\alpha,k}\right|\rho^\alpha \\
&\quad +
\sum_{\alpha,k: |\alpha|\leq s} \left|f_{e+l,\alpha,k} -f_{l,\alpha,k}\right|\rho^\alpha \\
&< \frac{\eps}{2} +  \frac{\eps}{2} = \eps.
\end{array}
\]
This holds for all $s$ so that 
\begin{equation}\label{c}
\|f -f_l\|_\rho = \sum_{\alpha,k} \left| f^{(\alpha,k)} -f_{l,\alpha,k}\right|\rho^\alpha <\eps
\quad {\rm for \; all\; } l \ge l_0.
\end{equation}
As $\|f\|_\rho \leq \|f-f_{l_0}\|_\rho +\|f_{l_0}\|_\rho \leq \eps + \|f_{l_0}\|_\rho
< \infty$, it follows $f \in B^N_\rho$. From (\ref{c}) 
it follows that $\lim_{e\to \infty}f_e = f$. Thus  $B^N_\rho$ is complete.
\end{proof}

\begin{proof}[Proof of Lemma \ref{L2}]
We may assume that $\LC(f_i)=1$ for all $i$.
We write $x^{\alpha(i)}e_{k(i)}$ for the leading term of $f_i$, and
we write $f_i = \sum_{\alpha,k} f_{\alpha,k}^{(i)} x^\alpha e_k$, $i=1,...,m$.
 Moreover
we define
\[
f_i^{(1)} = \sum_{\substack{(\alpha,k): |\alpha| = |\alpha(i)|\\ \alpha \neq \alpha(i)}} f_{\alpha,k}^{(i)} x^\alpha e_k, \quad f_i^{(2)} = \sum_{(\alpha,k): |\alpha| >
|\alpha(i)|} f_{\alpha,k}^{(i)} x^\alpha e_k.
\]
Obviously $\tail(f_i) = f_i^{(1)} + f_i^{(2)}$.

Given $\eps >0$, we will choose $\rho=(\rho_1,\ldots,\rho_n)$ so small that 
\begin{equation}\label{d}
\|f_i^{(1)}\|_\rho < \frac{\eps}{2}\|\LM(f_i)\|_\rho=\frac{\eps}{2}\|x^{\alpha(i)}\|_\rho.
\end{equation}
Let 
\[
M_i:= \{(\alpha,k): |\alpha| = |\alpha^{(i)}|, \; 
\alpha \neq \alpha^{(i)}\; {\rm and} \; f_{\alpha,k}^{(i)} \neq 0\}, \quad N_i := \# M_i.
\]
Then $f_i^{(1)}=\sum_{(\alpha,k)\in M_i} f_{\alpha,k}^{(i)} x^\alpha e_k$.
 and the inequality (\ref{d}) will follow from 
\begin{equation}\label{e}
\|f_{\alpha,k}^{(i)}x^\alpha e_k\|_\rho < \frac{\eps}{2N_i}\|\LM(f_i)\|_\rho=\frac{\eps}{2N_i}\|x^{\alpha(i)}\|_\rho
\quad \text{for all} \; (\alpha,k) \in M_i.
\end{equation}
To prove (\ref{e}) we construct $\rho_1,\ldots, \rho_{n}$ by descending induction.
We start with $\rho = (\tfrac{1}{2}, \ldots, \tfrac{1}{2})$.
 Write $\alpha = (\alpha_1,\ldots, \alpha_{n})$
and $\alpha(i) =  (\alpha_1(i),\ldots, \alpha_{n}(i))$.
The inequality $x^\alpha < x^{\alpha(i)}$  means
that there exists a $l = l(i,\alpha)$ such that
 \begin{equation}\label{f}
\alpha_1 = \alpha_1(i), \dots, \alpha_{l-1} = \alpha_{l-1}(i) \text{ and
  } \alpha_l > \alpha_l(i).
\end{equation}
Suppose that $\rho_\ell, \ldots, \rho_{n}$ are already defined 
such that (\ref{e}) holds for all $\alpha, i$ with $l(i,\alpha) \ge \ell$.
Because of (\ref{f}) the formula
(\ref{e}) will still hold
 after modifying $\rho_1,
\ldots, \rho_{\ell -1}$. 
Let $s \in \{1,\ldots, \ell -1\}
$ be the maximum of all $l(i,\alpha)$ for all $i$
and $(\alpha,k) \in M_i$, such that $l(i,\alpha) \leq \ell -1$. 
Consider  $i$ and $\alpha$ such that $s = l(i,\alpha)$.
We do not modify $\rho_{s+1}, \ldots, \rho_{n}$. We 
 can choose ${\rho}_s $ such that 
\begin{equation}\label{g}
{\rho}_s^{\alpha_s-\alpha_s(i)} < \dfrac{\varepsilon}{2N_i|f_{\alpha,k}^{(i)}\|}
    \cdot {\rho}_{s+1}^{\alpha_{s+1}(i)-\alpha_{s+1}}\cdot \cdots 
    {\rho}_{n}^{\alpha_{n}(i) - \alpha_{n}}.
\end{equation}
This we can do as $\alpha_s > \alpha_s(i)$.
So we may assume that 
 (\ref{e}) holds for all $i, \alpha$ 
with $s = l(i,\alpha)$. This gives the induction step. 
It remains to prove that for a suitable defined $\rho$
$$\|f_i^{(2)}\|_\rho < \frac{\eps}{2}\|\LM(f_i)\|_\rho=\frac{\eps}{2}\|x^{\alpha(i)}\|_\rho.$$
For $\lambda $ with $0 < \lambda \leq 1$ let $\lambda \circ \rho=(\lambda\rho_1,\ldots,\lambda\rho_n)$.
By definition of the norm  we have the equalities
$\|x^{\alpha(i)}\|_{\lambda \circ \rho} = \lambda^{|\alpha(i)|}\|x^{\alpha(i)}\|_\rho$,
 $\|f_i^{(1)}\|_{\lambda \circ \rho} = \lambda^{|\alpha(i)|}\|f_i^{(1)}\|_\rho$
and 
the inequality
 $\|f_i^{(2)}\|_{\lambda \circ \rho} \leq \lambda^{|\alpha(i)|+1}\|f_i^{(2)}\|_\rho$. Therefore, if we take
\[
\lambda <\frac{\eps\|x^{\alpha(i)}\|_\rho} {2\|f_i^{(2)}\|_\rho} \quad {\rm for \; all}\; i,
\]
we get 
\[
\|\tail(f_i)\|_{\lambda \circ \rho} \leq
\|f_i^{(1)}\|_{\lambda \circ \rho}+ \|f_i^{(2)}\|_{\lambda \circ \rho}
\leq
 \lambda^{|\alpha(i)|}\left( \frac{\eps}{2}\|x^{\alpha(i)}\|_\rho 
+\lambda \|f_2^{(i)}\|_\rho \right) \leq
\eps \|L(f_i)\|_{\lambda \circ \rho}.
\]
Thus the lemma holds with $\lambda \circ \rho$. 
\end{proof}

\section{Grauert's Approximation Theorem} \label{sec:4}
We start with "Cartan's lemma", which is an important step in Grauert's proof of his approximation theorem. As it generalizes straightforwardly to convergent power series over any real valued field (cf. \cite[Lemma 8.2.4]{dJP00}), the proof is omitted here.
It says (roughly speaking) the following. Given  a linear system of equations
$$\overset{r}{\underset{i=1}{\sum}} a_i  Z_i + 
\overset{r}{\underset{i=1}{\sum}} b_j Y_j = C$$
with given
$a_1, \dots, a_r$, $b_1, \dots, b_t \in K\{x\}^N$, $C\in  (K\{x\}[s])^N$ (with components homogeneous polynomials of fixed degree in $s$) and $Z_i, Y_j$ indeterminates. 
Assume that there exists a solution ${\sum} a_i  \bar{z}_i + {\sum} b_j \bar{y}_j = C$ with  $\bar{z}_i \in K\{x\}[s]$ and  $\bar{y}_j \in K[s]$. Then there exist a solution 
 ${\sum} a_i  {z}_i + {\sum} b_j {y}_j = C$ such that the norms of 
${z}_i \in K\{x\}[s]$ and ${y}_j \in K[s]$ can be bounded by the norm of $C$.\\

For $\rho \in (\R_{>0})^n$ and $\tau\in (\R_{>0})^l$ let $$B_{(\rho,\tau)}= \{f\in K[[x,s]]\text{  } | \text{  } \|f\|_{(\rho,\tau)}<\infty\},$$
with $x=(x_1,\ldots,x_n)$ and $s=(s_1,\ldots,s_l)$. 

\begin{Lemma}[Cartan's Lemma]\label{car}
 Fix $\bar\rho\in  (\R_{>0})^n$ and $\bar\tau\in (\R_{>0})^l$.  Let $a_1, \dots, a_r$, $b_1, \dots, b_t \in B^N_{\bar\rho}\subset K\{x\}^N$.  There exists a $\rho \in(\R_{>0})^n$, $\rho
 \le \bar{\rho}$, and $L = L(\rho) \in \R_{>0}$ with the following property:

For all $\tau \in(\R_{>0})^l$ with $\tau \leq \bar{\tau}$ and 
for all  $C \in B^N_{(\rho,\tau)}$, $C = \sum_{|\alpha|= e,k} C^{(\alpha,k)}
  s^\alpha e_k$ with $C^{(\alpha,k)} \in K\{x\}$ and for all homogeneous polynomials  $\bar{y}_1, \dots,
\bar{y}_t\in K[s]$ and $\bar{z}_1,\ldots \bar{z}_r \in K\{x\}[s]$ of degree $e$ in $s$ such that
$$\overset{r}{\underset{i=1}{\sum}} a_i
  \bar{z}_i + \overset{t}{\underset{j=1}{\sum}} b_j \bar{y}_j = C,$$
 there exist homogeneous polynomials $z_i\in K\{x\}[s]$, $i = 1,
\dots, r$ and $y_j \in K[s]$,  $j= 1, \dots, t$ of degree $e$ in $s$, such that
$$\overset{r}{\underset{i=1}{\sum}} a_i z_i +
  \overset{t}{\underset{j=1}{\sum}} b_j y_j = C,$$
 and $\|z_i\|_{(\rho,\tau)} \le L \|C\|_{(\rho,\tau)}$ and
  $\|y_j\|_\tau  \le L\|C\|_{(\rho,\tau)}$ for all $i$ and $j$,
  where the constant $L$ does not depends on $C$, $e$ and $\tau$.
\end{Lemma}

This result and the results of the previous section are used to generalize Grauert's approximation theorem to arbitrary real valued fields $K$.
For the proof of the approximation theorem, we require the following technical lemma; note that the proof of \cite[Lemma 8.2.5]{dJP00} carries over directly to our setting:

\begin{Lemma}\label{tech}
Let $Y=(Y_1,\ldots,Y_t)$ and 
  $F(x,s,Y) \in B^N_{(\rho, \bar{\tau},\sigma)}\subset K \{x,s,Y\}^N$,  $(\rho, \bar{\tau},\sigma) \in \R^{n+l+t}$.  Let $I \subset \C\{s\}$ be an
  ideal.  Then for all
  $\varepsilon > 0 $, there exists an $L = L(\varepsilon) > 0 $
and a $\tau \in (\R_{>0})^l$, $\tau \leq \bar{\tau}$, 
with the following property.\\
  Let   $ e \in \N$ and $y ^{(e)} (x,s) = (y^{(e)}_1 (x,s),
  \dots, y^{(e)}_t (x,s))\in K\{x\}[s]^t$, $y_i^{(e)}$ in normal form with respect to $I$ of degree $\leq e$ in $s$.
Let $F_e \in K\{x,s\}^N$ be the degree $e+1$ part
in $s$  of the normal form 
of $F(x,s,y^{(e)}(x,s))$ with respect to $I$.
  
If $\| y_\nu^{(e)}(x,s) \|_{(\rho, {\lambda\circ \tau})} \leq L\; {\rm for} \;
\nu =1, \ldots,t$ and some $\lambda$,  $0 < \lambda \leq 1$, then

\[
\|F_e \|_{(\rho, {\lambda\circ\tau})} \leq L \varepsilon.
\]
\end{Lemma}

We are now prepared to formulate and prove Grauert's Approximation Theorem.\\
  We consider four sets of variables, $x = (x_1, \ldots, x_n)$, $s =
  (s_1, \ldots, s_l)$, $Y = (Y_1, \ldots , Y_p)$ and $Z = (Z_1,
  \ldots, Z_q)$. 

\begin{Definition}\label{def.sol}
Let $I \subset K\{s\}$ be an ideal and let $\fm:=\langle s_1,\ldots,s_l\rangle\subset K\{s\}$.  Consider an element $F = (F_1, \ldots, F_N)  \in K\{x,s,Y,Z\}^N$.

(1) A {\em solution of order $e$} of the equation $F \equiv 0 \; {\rm mod} \;I$
  is a pair \\
\noindent $(y, z) \in K[s]^p \times K\{x\}[s]^q$ such that
\[
F(x,s, y(s), z(x,s)) \equiv 0 \quad {\rm mod} \; (I+ \fm^{e+1}) \cdot
K\{x,s\}^N.
\]

(2) An {\em analytic solution} of the equation $ F \equiv 0$ mod $I$ is a
  pair $(y, z) \in K\{s\}^p \times K\{x,s\}^q$ 
  such that
\[
F(x,s, y(s), z(x,s)) \equiv 0 \quad {\rm mod} \; I\cdot K\{x,s\}^N, 
\]
\end{Definition}
We can now formulate the theorem.

\begin{Theorem}[Grauert's Approximation Theorem for real valued fields]
\label{grauertapprox}
Let $e_0 \in \N$. Suppose that the system of equations 
$F \equiv 0 \; {\rm mod} \; I$
has
a solution $(y^{(e_0)}, z^{(e_0)})$ of order $e_0$.\\  Suppose, moreover,
that for all $ e \ge e_0$ every solution 
$(y^{(e)},z^{(e)})$ of order $e$  with
$y^{(e)} \equiv y^{(e_0)} \mod
{\fm}^{e_0+1}$, 
and $z^{(e)} \equiv z^{(e_0)} \mod {\fm}^{e_0+1}$,
extends to a solution  $(y^{(e)} +
u^{(e)}, z^{(e)} + v^{(e)})$ of order $e+1$,
 with
$u^{(e)}\in K[s]^p$ and $v^{(e)} \in K\{x\}[s]^q$ homogeneous of
degree $e+1$ in $s$.

Then the system of equations $F \equiv 0$ mod $I$ has
an analytic solution $(y, z)$, with $y \equiv y^{(e_0)}$ mod
${\fm}^{e_0+1}$, and $z \equiv z^{(e_0)}$ mod ${\fm}^{e_0+1}$.
\end{Theorem}
Note that solutions of order $e$  are (1-step) nested as $y^{(e)}$ depends only on $s$ and $z^{(e)}$ on $x$ and $s$. The assumption in the theorem that every solution modulo $e$ extends to a solution modulo $e+1$ leads obviously to a formal nested solution. 
The assumption in Grauert's approximation theorem, that {\em every} solution modulo $e$ extends to a solution modulo $e+1$, is stronger than  in the nested  Artin Approximation Theorem \ref{thm.nestapprox}, where only the existence of {\em one} formal solution is assumed. But the claim of the theorem, that an analytic nested solution exists, is also stronger, since the nested  Artin Approximation Theorem does not hold in the analytic category by \cite{Ga71}.\\

\begin{proof}[Proof of Theorem \ref{grauertapprox}]
The assumption that every solution modulo $e$ extends to a solution modulo $e+1$ leads to a formal solution but not necessarily to a convergent one. We have to prove that for a suitable choice of the extension the solution will be
convergent. We cannot use Artin's nested approximation since nested approximation is wrong in general in the
analytic case.\\
  The first step of the proof is a reduction to the case  that $y^{(e_0)}$ and
  $z^{(e_0)}$ are in $\fm$ and
  $F\bigr|_{s=0} \in K\{x\}^N$.
Let
\[
y^{(e_0)} = c + \sum s_\nu ( c_\nu +y^{(e_0)\prime}_\nu) \; \quad
  z^{(e_0)} = d + \sum s_\nu(d_\nu +z^{(e_0)\prime}_\nu).
\]
Here $c,d, c_\nu$ and $d_\nu$ are in $ K\{x\}$. Moreover 
  $y^{(e_0)\prime}_\nu = (y^{(e_0)\prime}_{1,\nu},\ldots,y^{(e_0)\prime}_{p,\nu})$,
  $z^{(e_0)\prime}_\nu = (z^{(e_0)\prime}_{1,\nu },\ldots,z^{(e_0)\prime}_{q,\nu })$ with
  $y^{(e_0)\prime}_{i,\nu } \in \fm $, $z^{(e_0)\prime}_{i,\nu } \in \fm \cdot K\{x,s\}$.
 Define new variables  $Y' = \{Y'_{i\nu }\},
  Z' = \{Z'_{i\nu }\}$, and define 
\[
F^\prime(x,s,Y', Z') := F(x,s,c + \sum s_\nu(c_\nu + Y'_\nu), d +
  \sum s_\nu(d_\nu+ Z'_\nu) ).
\]
  Now $F'$ has a solution $(y^{(e_0)'}, z^{(e_0)'})$ of order
$e_0$.  Replacing $F, y^{(e_0)}$, $z^{(e_0)}$ by $F', y^{(e_0)'},
z^{(e_0)'}$ proves the claim.

In the next step we reduce to the case that our solutions are in normal form with respect to $I$.

 If $(
y, z)$ is a solution of $F \equiv 0$ modulo $I$, then it follows
immediately that $(\NF(y,I), \NF(z,I))$ (see Definition \ref{def.NF}) is also a solution of $F \equiv
0$ modulo $I$. The same is true for solutions of order $e$,
as obviously $\sum a_\alpha s^\alpha$  being in normal form
implies $\sum_{|\alpha| \leq e} a_\alpha s^\alpha$ is in normal form.

  Let 
$(\bar{\rho}, \bar{\tau}, {\sigma}) \in (\R_{>0})^{n+l+p+q}$ such that $F
\in B^N_{(\bar{\rho}, \bar{\tau}, {\sigma})}$.  We consider the elements
\[
\frac{\partial F}{\partial Y}_{i}\bigl|_{s = 0} \in K\{x\}^{N} 
\text{  and } 
\frac{\partial F}{\partial Z}_{j}\bigl|_{s = 0} \in K\{x\}^N,
\quad {\rm for } \ i=1, \dots, p \;{\rm and} \;j=1,
\dots, q.
\]
By Cartan's
 Lemma \ref{car} we can choose  a $\rho \in \R^{n}_+$, $\rho \le
\bar{\rho}$ and $M = M(\rho) \ge 1$ depending on $\tfrac{\partial F}{\partial
  Y}\bigl|_{s = 0}$ and $\tfrac{\partial F}{\partial ZZ}\bigr|_{s = 0}$
with the following property: Let $G \in B^N_{(\rho, \bar{\tau})}$, and
$\bar{u}_1, \ldots, \bar{u}_p \in K[s]$ and $\bar{v}_1, \ldots,
\bar{v}_q \in K\{x\}[s]$ such that
\[
G + \sum_{i=0}^p\bar{u}_i \frac{\partial F}{\partial Y_i}\bigl|_{s =
  0} + \sum_{i=0}^q\bar{v}_i \frac{\partial F}{\partial Z_i}\bigr|_{s
  = 0}=0,
\]
and $G,\, \bar{u}_1, \dots, \bar{u}_p$, $\bar{v}_1, \dots,
\bar{v}_q$ are homogeneous with respect to $s$ of degree $e + 1$
for $e \ge e_0$.
Then there exist $u_1, \ldots, u_p \in K[s]$ and 
$v_1, \ldots, v_q \in K\{x\}[s]$ satisfying this equation being
homogeneous as before, where, moreover, 
\[
\|u_i\|_\tau, \; \|v_i\|_{(\rho, {\tau})} \leq M \|G\|_{(\rho,{ \tau})}
\;{\rm for\;  all} \; \tau \leq \bar{\tau}.
\]
We take $\varepsilon = \tfrac{1}{M}$. Let $L = L(\varepsilon)$
and $\tau \leq \bar{\tau}$  as in 
Lemma \ref{tech}, applied to $F(x,s,Y, Z)$ and let $(y^{(e_0)},
z^{(e_0)})$ be a solution of order $e_0$.  Because $y^{(e_0)} \in \fm$ and
$z^{(e_0)} \in \fm K\{x,s\}$ we can choose,
by scaling,  the polyradius
$\tau$ so small that, moreover,
$\|y^{(e_0)}\|_\tau<L$ and $ \|z^{(e_0)}\|_{(\rho,\tau)} < L$.

We will use induction to construct 
for all $ e \ge e_0+1$ solutions $(y^{(e)}, z^{(e)})$ of order
$e$ with $ (y^{(e)}, z^{(e)}) \equiv (y^{(e-1)}, z^{(e-1)})$
modulo ${\fm}^{e}$ and with the estimates
\[
\|y_e\|_\tau, \quad \|z_e\|_{(\rho, \tau)}   <L.
\]
Here $y_e$ and $z_e$ are the degree $e$ parts in $s$
of $y^{(e)}$ and $z^{(e)}$. 
Let $(y^{(e)}, z^{(e)})$ be a solution of order $e$,
such that $(y^{(e)}, z^{(e)})$
are polynomials in $s$ of degree $\leq e$, in normal form
with respect to $I$ and $\|y_e\|_\tau \leq L$, $\|z_e\|_{(\rho, 
\tau)} \leq L.$
   Let
$F_e$ be the degree $e+1$ part in  $s$ of the normal form of
$F(x,s,y^{(e)},z^{(e)})$.  By Lemma \ref{tech}, $\|F_e\|_{(\rho,
  \tau)} < \varepsilon L = \tfrac{L}{M}$.  By assumption, there exist
$\bar{u}^{(e)}$ and $\bar{v}^{(e)}$
homogeneous  of degree $e+1$ in $s$
and in normal form with respect to $I$  such that
$(y^{(e)} + \bar{u}^{(e)}, z^{(e)} + \bar{v}^{(e)})$ is a
solution of order $e+1$.  By Taylor expansion
and taking the degree $e+1$ part of the normal form we obtain 
\[
F_e + \sum_{i=1}^p \dfrac{\partial F}{\partial Y_i}|_{s =
  0}\bar{u}_i^{(e)} + \sum_{i=1}^q\dfrac{\partial F}{\partial Z_i}|_{s
  = 0}\bar{v}_i^{(e)}=0. 
\]
By the choice of $M$ there exist
such $u_i^{(e)}$ and $v_i^{(e)}$, having moreover the estimates
\[
\|u_i^{(e)}\|_\tau, \; \| v_i^{(e)}\|_{(\rho, \tau)} < M \|F_e\|_{(\rho, \tau)} < M\dfrac{L}{M} =L.
\]
Now define $y^{(e+1)}:= y^{(e)} + u^{(e)}$ and
 $z^{(e+1)}:= z^{(e)} + v^{(e)}$. 
Then $y_{e+1} =  u^{(e)}$ and $z_{e+1}= v^{(e)}$,
and the desired estimate holds.

We now look at the formal solution $y= \sum_{e= 0}^\infty y_e$,
$z= \sum_{e= 0}^\infty z_e$. Then,
because $y_e$ and $z_e$ are homogenous of degree $e$ with
respect to $s$, and $\|y_e\|_\tau < L$,  $\|z_e\|_{(\rho,\tau)} < L$
we obtain
\[
\|y\|_{\tau/2} < \sum_{e=0}^\infty \left(\tfrac{1}{2}\right)^e \cdot L
= 2L,\quad
\|z\|_{(\rho, \tau/2)} < \sum_{e=0}^\infty \left(\tfrac{1}{2}\right)^e \cdot L 
= 2L,
\]
showing the convergence of $y$ and $z$.
This finishes the proof of Grauert's Approximation Theorem. 
\end{proof}

\section{Unfoldings and the Splitting Lemma}\label{sec5}

Let $K$ be again a real valued field and $K\{x\}$ the convergent power series ring over $K$ as defined in Section \ref{sec:1}.\\
As an application of the Approximation Theorem \ref{grauertapprox} we prove first in Theorem \ref{thm.suunf} the existence of a semiuniversal unfolding of an isolated hypersurface singularity $f \in \fm^2 \subset K\{x\}$, $\fm =\langle x\rangle=\langle x_1,...,x_n\rangle$, allowing a translation in the source. We use this  to prove the splitting lemma in singularity theory in arbitrary characteristic (Theorem \ref{thm.split}). Not only the result is new, but also the method. It is interesting that we use a determinacy bound for isolated singularities (Proposition \ref{prop.rdet2}) (in a perhaps unexpected way) to prove the splitting lemma for {\em non-isolated} $f \in K\{x\}$.
We would like to point out that our proof of the splitting lemma is much easier than the fairly complicated original proof by R. Thom (for differentiable functions) in his book on elementary catastrophes.

In the last section (Section \ref{sec6}) we deduce from the approximation theorem the existence of a semiuniversal deformation for an arbitrary analytic algebra $K\{x\}/I$ with isolated singularity, generalizing  Grauert's theorem for complex analytic singularities in \cite{Gr}. 

The results of this and the last section encompass and go beyond previously known results 
(for further results on unfoldings see \cite[Chapter 3.3 and 3.5 (by D. Kerner)]{GLS25}).\\

The notion of unfolding in the form we need, is as follows:

\begin{Definition} \label {def.unf}
 Let $f \in \fm \subset K\{x\}$. 
 \begin{enumerate}
 \item A {\em (p-parameter) unfolding} of $f$ is a power series 
 $F \in K\{x,s\}$, $s= (s_1,...,s_p)$, such that $F(x, 0) =f$. 
If $F(x,s) = f(x)$ we say that $F$ is the {\em constant unfolding} of $f$. $K\{s\}$ is called (the algebra of) the {\em parameter space} of $F$.
  \item Let $F \in K\{x,s\}$ and $G \in K\{x,t\}$, $t = (t_1,...,t_q)$,
 be two unfoldings of $f$. 
 A {\em right-left morphism from $G$ to $F$} is given by a tripel
 $(\Phi, \phi,\lambda),$ with
 \begin{enumerate}
  \item [(i)] $\Phi : K\{x,t\} \to K\{x,t\}$ a morphism, satisfying\\
  $\Phi(x_i) =: \Phi_i \in K\{x,t\}$, $\Phi_i(x,0) = x_i$, $i=1,...,n$,\\
  $\Phi(t_j) = t_j, \ j=1,...,q$,
 \item [(ii)] $\phi : K\{s\} \to K\{t\}$, a morphism, $\phi(s_j) =: \phi_j(t) \in K\{t \}, \ j=1,...,p$,   and   
\item [(iii)] $\lambda \in  K\{y,t\}$, $y=(y_1)$, $\lambda(y,0) = y$,   such that 
 $$\Phi(G)(x,t)= G(\Phi(x,t),t)  = \lambda(F(x,\phi(t)),$$
with $\Phi(x,t):=(\Phi_1(x,t),...,\Phi_n(x,t))$ and $\phi(t):=(\phi_1(t),...,\phi_p(t))$.
   \end{enumerate}
 If this holds, we say that  $G$ is {\em right-left induced}  from $F$.  If $\lambda$ is a translation, that is $\lambda(y,t) = y + \alpha(t),  \alpha \in K\{t\}, \alpha(0)=0$, $G$ is called {\em right induced}  from $F$.
\item  An unfolding $F$ of $f$ is called {\em  right-left complete}  (resp. {\em right complete})\,\footnote{In the literature also "versal"
is used instead of "complete". We prefer to use versal only in the sense of Definition \ref{def.def}.}  if any unfolding of $f$ is induced from $F$ by a suitable right-left (resp. right) morphism. \\
$F$ is called {\em right-left versal} (resp. {\em right versal}) if in addition the lifting property from Definition \ref{def.def} (5) holds for right-left (resp. right) induced unfoldings.
A versal unfolding is called {\em semiuniversal} or {\em miniversal} if the cotangent map of the base change map is uniquely determined (see Definition \ref{def.def} (5) below). 
 \end{enumerate}
 \end{Definition} 

\begin{Remark}\label{rm.unf}{\em
Since the Jacobian matrix of $\Phi$ at 0 is the identity in Definition \ref{def.unf}, $\Phi$ is in fact an automorphism of $K\{x,t\}$. Hence, if $G$ is right-left induced from $F$, then $G(x,t)$ is right-left equivalent to 
$H(x,t)$, where $H(x,t) = F(x,\phi(t))$ is obtained from $F$ by the base change $\phi$.

Usually the notion right equivalence is used if the translation is trivial, that is, $\alpha=0$, but the case  of a non-trivial translation is essential in the proof of Theorem \ref{thm.split}.
The  translation $\alpha$  is  introduced to take care of the constant terms of $ F(x,{\phi(t)})$ for varying $t$. 

The following theorem is known for real and complex analytic germs with proofs that  use the existence and uniqueness of the solution of an ODE. This cannot be extended to positive characteristic. Our proof uses Grauert's Approximation Theorem \ref{grauertapprox} and is rather simple.}
 \end{Remark}

\begin{Theorem}\label{thm.suunf}
Let $f\in \fm^2 \subset K\{x\}$ and assume that $\dim_K\fm/\langle\frac{\partial f}{\partial x_1},\ldots,\frac{\partial f}{\partial x_n}\rangle<\infty$.
Let $g_1,\ldots,g_p\in \fm$ be representatives of a generating system (resp. a basis) of $\fm/\langle\frac{\partial f}{\partial x_1},\ldots,\frac{\partial f}{\partial x_n}\rangle$ 
and $s=(s_1,\ldots,s_p)$ new variables.
Then 
$$F(x,s) := f(x)+\sum_{i=1}^pg_i(x)s_i\in K\{x,s\}$$ is a right-versal (resp. right-semiuniversal) unfolding  of $f$.
\end{Theorem}

 \begin{proof}
We only prove completeness; versality can be proven in a similar way, but is more complicated in terms of notation.

 Let $G(x,t)\in K\{x,t\} $ be an arbitrary unfolding of $f$, $t=(t_1,\ldots,t_q)$.
 We have to prove that there exists $\phi(t)\in \langle t\rangle K\{t\}$, an automorphism  $\Phi$ of $K\{x,t\}$, $\Phi = (\Phi_1,...,\Phi_n)$, $\Phi_i = \Phi(x_i)$,
with $\Phi_i(x,0) = x_i$,
$\Phi(t_j) = t_j$, and $\alpha \in K\{t\}$ with  $\alpha(0)=0$,
such that
\begin{equation}\label{def1}
\Phi(G)=G(\Phi(x,t),t) = F(x,\phi(t)) +\alpha(t).
\end{equation}
A solution of order $e$ (in $t$) of (\ref{def1}) is a triple $(\Phi,\phi,\alpha) \in K\{x\}[t]^n \times K[t]^p  \times K[t]$  such that 
(\ref{def1}) holds mod $\langle t \rangle^{e+1}= \langle t \rangle^{e+1}K\{x,t\}$.

Since $\Phi(x,0) =x$, $\phi(0)= 0$, $\alpha(0)=0$  and $G(x,0)=f(x)=F(x,0)$ we obviously have a solution of (\ref{def1}) of order $0$. 
To apply Theorem \ref{grauertapprox}, we have to show that every solution 
$(\Phi,\phi,\alpha)$ of $(\ref{def1})$ of order $e$ in $t$ can be extended to a solution
$(\Phi',\phi',\alpha')$ of $(\ref{def1})$ of order $e+1$. 
\\
Having the solution $(\Phi,\phi,\alpha)$ of order $e$, the difference
$G(\Phi(x,t),t) - F(x,\phi(t)) - \alpha(t)$ is a power series of order $\ge e+1$ in $t$. The homogeneous part of degree $e+1$ in $t$ of
$G(\Phi(x,t),t) - F(x,\phi(t)) -\alpha(t)$ can thus be written as
\[
 G(\Phi(x,t),t) - F(x,\phi(t)) -\alpha(t) = \sum_{| \nu | = e+1}t^{\nu}h_{\nu}(x) \mod \langle t\rangle^{e+2}
\]
with $h_{\nu} \in K\{x\}$. 
Define $g_0:=1$. By assumption $g_0,...,g_p$ generate $K\{x\}/\langle\frac{\partial f}{\partial x_1},\ldots,\frac{\partial f}{\partial x_n}\rangle$ and therefore we can write for all $\nu$ 
\[
h_{\nu} =  \sum_{j=0}^{p}b_{j \nu}g_j + \sum_{j=1}^nh_{\nu j}
\frac{\partial f}{\partial x_j}
\]
with $b_{j \nu } \in K$, $j \ge 1$,  and $h_{\nu j} \in K\{ x\}$.
We define 
\begin{align*}
\Phi'_j &= \Phi_j - \sum_{|\nu | = e+1}{}{h_{\nu j}}t^\nu \; \quad
j=1,\ldots,n\\
\phi'_j &= \phi_j + \sum_{|\nu | = e+1}b_{j \nu}t^\nu\;\quad j = 1, \ldots, p,\\
\alpha' &= \alpha +\sum_{|\nu | = e+1}b_{0 \nu}t^\nu.
\end{align*}
As $G(x,0)= f$ and $\Phi(x,0)=x$ we obtain 
$$\frac{\partial{(G(\Phi(x,t),t))}}{\partial{x_j}}\equiv \frac{\partial f}{\partial x_j} \mod \langle t\rangle .$$
This implies using Taylor's formula (Remark \ref{rm.unf}) 
\begin{align*}
G(\Phi'(x,t),t)&=  G(\Phi(x,t),t)-\sum_{j=1}^n \frac{\partial G}{\partial x_j}(\Phi(x,t),t)(\sum_{|\nu | = e+1}{}{h_{\nu j}}t^\nu) \mod \langle t\rangle^{e+2}\\
&=  G(\Phi(x,t),t)-\sum_{|\nu | = e+1}(\sum_{j=1}^nh_{\nu j}\frac{\partial f}{\partial x_j} )t^\nu \mod \langle t\rangle^{e+2}.
\end{align*}
On the other hand, by Taylor and the definition of $F$,
\begin{align*}
F(x,\phi'(t))& = F(x,\phi(t)) +\sum_{j=1}^p(\frac{\partial F}{\partial s_j}(x,\phi(t))(\sum_{|\nu | = e+1}b_{j \nu}t^\nu) \mod \langle t\rangle^{e+2}\\
&=F(x,\phi(t))+\sum_{|\nu | = e+1}(\sum_{j=1}^p b_{j\nu}g_j)t^\nu \mod \langle t\rangle^{e+2}.
\end{align*}
We obtain
$$G(\Phi'(x),t)-F(x,\phi'(t))-\alpha'(t)\equiv 0 \mod \langle t\rangle^{e+2}.$$
Now we can apply Theorem \ref{grauertapprox} to obtain an analytic solution of (\ref{def1}).
This proves the theorem.
\end{proof}

\begin{Remark}\label{rm.unf}
{\em
(1) We can of course define unfoldings $G$ of $f$ over arbitrary parameter spaces $K\{t\}/J$, $J\subset K\{t\}$ an ideal, as  an element
$G(x,t)\in (K\{t\}/J)\{x\} $ with $G(x,0)=f$.\\
 Then the power series $F$ from Theorem \ref{thm.suunf} is also right-versal for such unfoldings $G$:
Namely, let $G(t,x)\in (K\{t\}/J)\{x\} $ be an unfolding of $f$, $t=(t_1,\ldots,t_q)$. 
 We have to prove that there exists $\phi(t)\in \langle t\rangle$,
an automorphism  $\Phi$ of $(K\{t\}/J)\{x\}$  with $\Phi(x)_{t=0} = x$ and
$\Phi(t) = t$ and
an $\alpha \in K\{t\}$ with  $\alpha(0)=0$, satisfying
\begin{equation}\label{def2}
 \Phi(G)= G(\Phi(x,t),t)  = F(x,\phi(t)) +\alpha(t).
\end{equation}
To see this, we take a representative of $G(t,x)$ in $K\{t,x\}$ which  is an unfolding of $f$ over $K\{t\}$. We apply  Theorem (\ref{thm.suunf}) to this representative and pass then to  $(K\{t\}/J)\{x\} $.

(2) At several places we use (the replacement of) the Taylor series in positive characteristic in the following form:\\
Let $f(x)= \sum\limits_{\left| {{\alpha}}  \right| \ge ord(f)} {c_\alpha  x^{{\alpha}}}\in K\{x\}$, $x=(x_1,...,x_s)$ and ${z}=(z_1,...,z_s)$ new variables. Then  
	\begin{align*}
      f({x+z}) =& \sum_{|\alpha|\ge \order(f)}c_\alpha(x+z)^\alpha
	=f({ x})+
	\sum\limits_{\nu = 1}^s {\frac{{\partial f({x})}}{{\partial x_\nu }}\cdot z_\nu }  + H,\\
H=&	\sum\limits_{\left| \alpha  \right| \ge \order(f)} {c_\alpha\cdot  
\Big ( 
{\sum\limits_{\left| \gamma  \right| \ge 2 \atop \gamma  \le \alpha } 
{\binom {\alpha _1 } { \gamma _1} \cdot\ldots\cdot
\binom {\alpha _s } { \gamma _s}{x}^{\alpha  - \gamma }{z}^\gamma} }
\Big )}.
\end{align*}
If $z=z(x)$ then $\order (H(x)) \ge  \ord(f) + |\gamma|  \order ({z(x)}-1) \ge \ord(f) + 2  \order ({z(x)}-1).$
$\gamma\le \alpha$ means that $\gamma_\nu\le\alpha_\nu$ for all $\nu$, and for $k\in \Z$
we have $k{ x}^{\alpha  - \gamma }{ z}^\gamma= 0$ if   $p\mid k$ and  $k {(\modular\hskip 3pt p)}{ x}^{\alpha  - \gamma }{ z}^\gamma$ if  $p\nmid k$. 
}
\end{Remark}

For the proof of the splitting lemma we need also a determinacy bound for an isolated singularity $f\in K\{x\}$.
Recall that $f$ is {\em right equivalent} to $g\in K\{x\}$ ($f \rsim g$)
if there exists an automorphism $\varphi$ of $K\{x\}$ such that 
$\varphi(f) = g$.  $f$ is {\em contact equivalent} to 
$g$ ($f \csim g)$
if there exists in addition a unit $u\in K\{x\}$ such that $u\varphi(f) = g$. For a non-negative integer $k$ we call $f$ {\em right $k$-determined} (resp. {\em contact k-determined}) if $f \rsim g$ (resp. $f \csim g$) for each $g$ with $f-g \in \fm^{k+1}$.

\begin{Proposition}\label{prop.rdet1}
Let $f\in \fm^2 \subset K\{x\}$ and assume that $\fm^{k+1}\subset \fm \langle\frac{\partial f}{\partial x_1},\ldots,\frac{\partial f}{\partial x_n}\rangle^2$ then $f$ is right $k$-determined.
\end{Proposition}
\begin{proof}
The proof is a direct consequence of Newtons Lemma (see \cite[Theorem 9.1.3]{dJP00}).
\end{proof}

\begin{Proposition}\label{prop.rdet2}
Let $f\in \fm^2 \subset K\{x\}$.\\
1. If $\fm^{k+2}\subset \fm^2 \langle\frac{\partial f}{\partial x_1},\ldots,\frac{\partial f}{\partial x_n}\rangle$ then $f$ is right
 $(2k-\ord(f)+2)$-determined. \\
2. If $\fm^{k+2}\subset \fm \langle f  \rangle  + \fm^2 \langle\frac{\partial f}{\partial x_1},\ldots,\frac{\partial f}{\partial x_n}\rangle$ then $f$ is contact $(2k-\ord(f)+2)$-determined.
\end{Proposition}

\begin{proof} The proof was first given in  \cite[Theorem 3]{BGM12} for formal power series. We give here a (shorter) proof for $K\{x\}$ using of Grauert's approximation theorem. We prove only statement 1., the proof of statement 2. is similar.\\
Let $N\geq 2k-\ord(f)+2=:e_0$ and assume that $f-g\in \fm^{N+1}$. We have to prove that there exists an automorphism $\varphi=(\varphi_1,\ldots,\varphi_n)$
of $K\{x\}$ such that $f(\varphi)=g$. We consider the equation 
$$F(x,y):=f(y)-g(x)=0, \ y=(y_1,\ldots,y_n),$$
and want to apply Theorem \ref{grauertapprox} to the equation $F=0$. Since $f-g\in \fm^{e_0+1}$,
this equation has the solution $\varphi^{(e_0)}(x_i)=x_i$ of order  $e_0$:
$$F(x,\varphi^{(e_0)})=0 \mod \fm^{e_0+1}.$$
Let $\varphi^{(e)}$ be a solution of $F(x,y)=0$ of order $e \ge e_0$ such that $\varphi^{(e)}\equiv \varphi^{(e_0)} \mod \fm^{e_0+1}$, that is , $\varphi^{(e)}(x_i) = x_i + h_i$ with $h_i \in \fm^{e_0+1}$.
By assumption we have 
$\fm^{k+2}\subset \fm^2j(f)$ and
$F(x,\varphi^{(e)})\subset \fm^{e+1}$.
Since $\fm^{e+1} = \fm^{e-k-1}\fm^{k+2} \subset \fm^{e-k+1}j(f)$,
we can write
$$F(x,\varphi^{(e)})=\sum_{i=1}^nb_i(x)\frac{\partial f}{\partial x_i}(x)$$
with $b_i\in \fm^{e-k+1}$. Now define
$\varphi^{(e+1)}$ by $\varphi^{(e+1)}_i=\varphi^{(e)}_i-b_i$. We obtain by Taylor's formula (see Remark \ref{rm.unf})
\begin{align*}
F(x,\varphi^{(e+1)})=F(x,\varphi^{(e)}-b)=F(x,\varphi^{(e)})-\sum_{i=1}^nb_i(x)\frac{\partial F}{\partial y_i}(x,\varphi^{(e)})+H(x)
\end{align*}
with $\ord(H) \ge \ord(f) + 2( \ord(b) -1) \ge \ord(f) + 2( e-k)
\ge e + (e_0-2k +\ord(f)) = e+2.$\\
Moreover
$\frac{\partial F}{\partial y_i}(x,\varphi^{(e)}) = \frac{\partial f}{\partial x_i}(\varphi^{(e)}) =   \frac{\partial f}{\partial x_i}(x+h) = \frac{\partial f}{\partial x_i}(x) +
 \sum_{j=1}^n h_j \frac{\partial^2 f}{\partial x_j\partial x_i}(x) + G(x)$ and we can estimate $\ord(b_i(\sum_{j=1}^n h_j \frac{\partial^2 f}{\partial x_j\partial x_i}(x) + G(x)) \ge e+2.$
 
This implies $F(x,\varphi^{(e+1)}) =0  \mod \fm^{e+1}$
and thus $\varphi^{(e+1)}$ is an extension of the solution $\varphi^{(e)}$ of order $e+1$. We can apply
Grauert's Approximation Theorem \ref{grauertapprox} and obtain a solution $\varphi$ of $F=0$. This proves statement 1.
\end{proof}

We formulate now the general splitting lemma for power series in $K\{x\}$, which was proved in \cite[Section 2 and 3]{GP25} under additional assumptions. Note that for $f \in \fm^2$ the rank of the Hessian matrix $H(f):=\Big( \frac{\partial^2 f}{\partial x_i\partial x_j}(0)\Big)_{i,j=1,\ldots,n}$ is invariant under right equivalence. 

We would like to mention that the proof for the existence of the splitting in characteristic 2 uses a new idea for {\em non-isolated singularities}.

\begin{Theorem}[Splitting lemma in any characteristic]
\label{thm.split}
Let $K$ be a real valued field and $f \in \fm^2 \subset K\{x\}$.
\begin{enumerate}
\item Let $char(K) \ne2$ and   $\rank H(f)=k$. Then 
 $$ f\ \rsim \ a_1x_1^2+\ldots+a_kx_k^2+g(x_{k+1},\dots,x_n) $$
   with $a_i \in K$,  $a_i \ne 0$, and
   \mbox{$g\in\fm^3$}. 
   $g$ is called the {\em residual
     part\/}\index{residual part} of $f$, it is uniquely determined up to   right equivalence in $K\{x_{k+1},...,x_n\}$.   
\item Let char$(K)=2$ and $\rank H(f)=2l$\ \footnote{In characteristic 2 the rank of the Hessian matrix is even.}. Then $f$ is right equivalent to 
$$ \sum_{i \text{ odd, } i=1}^{2l-1}(a_i x_i^2 + x_ix_{i+1} +
a_{i+1} x_{i+1}^2) + \sum_{i=2l+1}^{n}d_i x_i^2 + h(x_{2l+1},...,x_n),
$$
with  $a_i, d_i \in K$, $h\in \langle x_{2l+1},\ldots,x_{n} \rangle^3$. 
$g:=  \sum_{i=2l+1}^{n}d_i x_i^2  + h(x_{2l+1},\ldots,x_{n})$ is called the {\em residual part} of $f$, it is uniquely determined up to right equivalence in $K\{x_{2l+1},\ldots,x_{n}\}$.
\end{enumerate} 
\end{Theorem}

If the field $K$ is algebraically closed, the coefficients $a_i$ and $d_i$ can be made to $1$. More precisely, we have

\begin{Corollary} \label{cor.split}
Let $f$ be as in Theorem \ref{thm.split}.
\begin{enumerate}
\item Let $char(K)\ne 2$ and assume that $K$ coincides with its subfield $K^2$ of squares. Then, with $k=\rank H(f),$ 
  $$ f\ \rsim x_1^2+\ldots+x_k^2+g(x_{k+1},\dots,x_n). $$
$g$  uniquely determined up to  right equivalence in $K\{x_{2l+1},\ldots,x_{n}\}$. 
\item  Let $char(K) = 2$ and assume that quadratic equations are solvable in $K$. Then  $f$ is right equivalent to one of the following normal forms, with $g$ unique up to right equivalence in $K\{x_{2l+1},\ldots,x_{n}\}$, $2l=\rank H(f)$:
$$
\begin{array}{llll}
(a) & x_1x_2+x_3x_4+\ldots+x_{2l-1}x_{2l}+x_{2l+1}^2 &+g(x_{2l+1},\ldots,x_{n}), & 1\le  2l+1 \le n,\\
(b)& x_1x_2+x_3x_4+\ldots+x_{2l-1}x_{2l} &+ g(x_{2l+1},\ldots,x_{n}),  & 2 \le 2l \le n.
\end{array}
$$
\end{enumerate}
\end{Corollary}

\begin{proof}[Proof of Theorem \ref{thm.split}] 
1. The splitting lemma for char$(K) \ne2$ was proved in \cite[Theorem 2.1]{GP25} for formal power series  and for $f\in K\{x\}$ under the assumption that $f$ has an isolated singularity. Moreover, the statement was proved for  $f\in \C\{x\}$ or  $f\in \R\{x\}$  without assuming that $f$ has an isolated singularity (loc. cit. Theorem 2.4). The proof uses as an intermediate step the existence of a convergent semiuniversal unfolding (with non-trivial translation), a result basically due to Mather.
Since we proved the existence of a convergent semiuniversal unfolding in Theorem \ref{thm.suunf} for arbitrary real valued $K$, the proof given in \cite[Theorem 2.4]{GP25} works as well in our case and thus proves the splitting lemma in characteristic $\ne 2$.

2. The uniqueness of the residual part in characteristic  2 was already 
proved in \cite[Remark 3.6]{GP25}.  The existence of a splitting was proved in 
loc. cit. (Theorem 3.5) for formal and algebraic power singularities over any field. It remains to prove the existence for convergent power series over real valued fields. \\
By \cite[Theorem 3.1]{GP25}  (applied to the 2-jet of $f$) we can assume that
  $ f(x)= q + g$, where
$$ q(x_1,...,x_{2l}) = \sum_{i \text{ odd, } i=1}^{2l-1}(a_i x_i^2 +  x_ix_{i+1}+ a_{i+1} x_{i+1}^2)$$
and $g(x_1,...,x_n)=\sum_{i=2l+1}^{n}d_i x_i^2 +h(x_1,...,x_n)$,
with $a_i, d_i \in K$ and $h\in \fm^3 \subset K\{x\}$.\\
Setting $g'(x_1,...,x_{2l}):= g(x_1,...,x_{2l},0,...,0)$ we get
 $g(x_{1},\dots,x_n)=g'(x_1,...,x_{2l}) + \sum_{i=2l+1}^{n}d_i x_i^2 +h'(x_1,...,x_n)$, with 
 $h'= \sum_{i=2l+1}^n x_ih_i(x_1,...,x_n)$. Then $f = q+g = 
 q+ g'+ \sum_{i=2l+1}^{n}d_i x_i^2 + h'=f'+  \sum_{i=2l+1}^{n}d_i x_i^2+ h',$ with
  $$f'(x_1,...,x_{2l}):=  f(x_1,...,x_{2l},0,...,0) =q(x_1,...,x_{2l}) + g'(x_1,...,x_{2l}).$$
 Since $d(q)=1$ we get from Proposition \ref{prop.rdet2} that $q$ is 2-determined.
Since $g' \in \fm^3$, it follows that $f'$
 is right equivalent to $q$ by an automorphism $\varphi$ of 
 $K\{ x_1,...,x_{2l}\}$. Setting $\varphi_i(x_1,...,x_{2l}) := \varphi(x_i)$,
we have
 $$\varphi(f') = q(\varphi_1,...,\varphi_{2l})+g'(\varphi_1,...,\varphi_{2l}) = q(x_1,...,x_{2l}).$$
 Now define the automorphism  $\psi$ of $K\{ x_1,...,x_n\}$ by $x_i\mapsto \varphi_i(x_1,...,x_{2l})$ for
  $i=1,\dots,{2l}$, and \mbox{$x_i\mapsto x_i$} for $i>{2l}$. Then 
  $\psi(f') = \varphi(f') =q$ and
 we have
  $$\psi(f) = 
  q+\sum_{i=2l+1}^{n}d_i x_i^2 +\sum_{i={2l}+1}^n x_ih_i(\varphi_1,...,\varphi_{2l},x_{{2l}+1},...,x_n). $$
 Thus, after applying $\psi$, we may assume that $f= q +\sum_{i=2l+1}^{n}d_i x_i^2+ h'$,
 $h'= \sum_{i={2l}+1}^n x_ih_i(x_1,...,x_n).$
  That is $f(x_1,...,x_{2l},0,...,0) = q$, in other words, $f$ is an unfolding of $q$.\\
  Since $\langle \frac{\partial q}{\partial x_1},..., \frac{\partial q}{\partial x_{2l}}\rangle =\langle x_1,...,x_{2l}\rangle$, $q$ is a semiuniversal unfolding of itself by Theorem \ref{thm.suunf}, and 
the unfolding $f$ can be right induced from $q$.
That is,  there exists an automorphism  $\Phi =(\Phi_1,...,\Phi_n)$ of $K\{x\}$  with $\Phi_i=\Phi(x_i)$,
 $\Phi_i(x_1,...,x_{2l},0) = x_i$, $i=1,...,2l$,
  $\Phi(x_j) = x_j$, $j=2l+1,...,n$  
  and $\alpha \in K\{t\}$ with  $\alpha(0)=0$,
such that 
$$\Phi(f)=f(\Phi_1(x),...,\Phi_{2l}(x),x_{2l+1},...,x_n) = 
q(x_1,...,x_{2l}) +\alpha(x_{2l+1},...,x_n).$$ 
Now $\Phi(f)= \Phi(q)+\Phi(\sum_{i=2l+1}^{n}d_i x_i^2)+\Phi(h')
=q  +\sum_{i=2l+1}^{n}d_i x_i^2 + h$, where $h:=\Phi(h') \subset K\{x_{2l+1},...,x_n\}$. This proves statement 2.
\end{proof}

\section{Semiuniversal Deformation}\label{sec6}
We are now going to prove the existence of a convergent semiuniversal deformation for an analytic $K$-algebra\,\footnote{An {\em analytic $K$-algebra}, or just {\em analytic algebra},  is a $K$-algebra isomorphic to  $K\{x_1,...,x_n\}/I$ 
for some $n\ge 0$ and some ideal $I$. For properties of analytic algebras see \cite[Sections 1.1.1 and 1.1.2]{GLS25}} with isolated singularity, $K$ a real valued field, by applying Theorem \ref{grauertapprox}. Before we formulate it, we recall the definition of a (semiuni-) versal deformation (see \cite{GLS25} for a geometric formulation for complex germs and \cite{St03} for general cofibred groupoids.)

\begin{Definition} \label{def.def}Let $R=K\{x\}/I$ and $T=K\{t\}/J$ be two analytic $K$-algebras, $x=(x_1,...,x_n)$, $t=(t_1,...,t_q)$ with $I,J$ ideals and $K$ a real valued field. 
\begin{enumerate}
   \item  
   A {\em deformation of $R$ over $T$} is a flat morphism $\phi: T \to \kr$ of analytic $K$-algebras together with an isomorphism 
$\kr \otimes_T T/\langle t \rangle \to R$. The algebras $\kr$ (resp. $T$, resp. $\kr/\langle t \rangle \kr \cong R$) are called the (algebras of the) {\em total space} (resp. {\em base space}, resp. {\em special fiber}) of the deformation. We denote a deformation of $R$ by 
$(\phi,\iota): T\to \kr \to R$, with $\iota :\kr \to R$ the canonical projection, or just by $\phi: T\to \kr$. 
   \item 
   Let $(\phi',\iota'): T'\to \kr' \to R$ be a second deformation of $R$ over the analytic $K$-algebra $T'$. A {\em morphism} from $(\phi',\iota')$ to $(\phi,\iota)$ 
is given by a pair of morphisms $(\psi: \kr \to \kr', \varphi : T\to T')$
such that $\iota = \iota' \circ \psi$ and $\psi \circ \varphi = \phi' \circ \varphi$. Two deformations over the {\em same} base space $T$ are {\em isomorphic}  if there exists a morphism  $(\psi,\varphi)$ with $\psi$ an isomorphism and $\varphi = \id_T$.  
\item 
Let $(\phi,\iota): T\to \kr \to R$ be a deformation of $R$, 
$\varphi : T\to S$ a morphism of analytic algebras, and let the lower square of the commutative diagram
$$
\begin{xymatrix}{
&R&\\
 \kr \hat\otimes_T S \ar[ur]^{\varphi^*\iota}  && \kr \ar[ll]_{\tilde\varphi} \ar[ul]_\iota \\
S \ar[u]^{\varphi^*\phi} && T \ar[ll]^\varphi \ar[u]_\phi  }
\end{xymatrix}
$$
be the analytic pushout\,\footnote{\,The analytic pushout can be constructed as the analytic tensor product $\kr \hat\otimes S$ (see \cite[Kapitel III, § 5]{GR71})
modulo the ideal generated by $\{(\phi(t)r,s)-(r,\varphi(t)s), t\in T, r \in \kr, s \in S\}$. It is the algebraic counterpart to the geometric pullback or fiber product.} of $\phi$ and $\varphi$. 
Then the morphism $\varphi^*\phi$ is flat (see e.g. \cite[Proposition 1.1.87]{GLS25}) and
there is a natural map 
$\varphi^*\iota:  \kr \hat\otimes_T S \to R$  such that 
$$(\varphi^*\phi, \varphi^*\iota) : S \to \kr \hat\otimes_T S \to R$$
is a deformation of $R$ over $S$. 
$(\varphi^*\phi, \varphi^*\iota)$ is called the 
{\em deformation induced from $(\phi, \iota)$ by $\varphi$} and 
$\varphi$ is called the {\em base change map}.
$(\tilde \varphi,  \varphi)$ is a 
morphism of deformations from $(\varphi^*\phi, \varphi^*\iota)$ to
$(\phi, \iota)$. 
  \item    
  A deformation $(\phi,\iota): T\to \kr \to R$ of $R$ over $T$ is called {\em complete} if any deformation $(\psi,j): S\to \kq\to R$ of $R$ over some analytic algebra $S$ can be induced from $(\phi,\iota)$ by some base change map $T \to S$   (up to isomorphism of deformations of $R$ over $S$).
\item
 The deformation $(\phi,\iota)$ of $R$ over $T$ is called {\em versal} if it is complete and the following lifting property holds:\\
Let  $(\psi,j)$ be a given deformation of $R$ over $S$. Let $k: S\to S'$ a surjection and $\varphi': T\to S'$ a morphism of analytic algebras, such that 
the induced deformations  $(\varphi'^* \phi, \varphi'^* \iota)$ and 
$(k^*\psi,k^*j)$ over $S'$ are isomorphic. Then there exists a morphism $\varphi: T \to S$ such that $k\circ \varphi = \varphi'$ 
and $(\psi,j) \cong (\varphi^* \phi, \varphi^* \iota)$.
\item 
A versal deformation is called {\em semiuniversal} or {\em miniversal} if, with the notations from 5.,  the cotangent map of
$\varphi$, $\dot\varphi: \fm_T/\fm_T^2 \to \fm_S/\fm_S^2$, is uniquely determined by $(\phi,\iota)$ and $(\psi,j)$. 
\end{enumerate}
\end{Definition}
Versality of $(\phi,\iota)$ means that $(\phi,\iota)$ is not only complete but in addition that any deformation $(\psi,j)$ can be induced from $(\phi,\iota)$ by a base change that  extends a given base change  inducing  $(k^*\psi,k^*j)$ from 
 $(\phi,\iota)$. $(\phi,\iota)$ is called {\it formally versal} if the lifting property 5. holds for Artinian analytic algebras $S$ and $S'$.\\
 Property 5. is needed to construct a versal deformation (completeness is not sufficient): starting from the trivial deformation $T=K \to \kr =R$
 one extends this to bigger and bigger Artinian base spaces. This can be done since the so called 'Schlessinger conditions'  (a kind of formal version of the lifting property 5.) hold for the deformation functor. In the limit one gets finally to a formal object, which is formally versal. \\
 To get an analytic versal deformation Grauert formulates the lifting property 5.
as the solution of analytic equations. Constructing the liftings such that the assumptions of Grauert's Approximation Theorem hold, leads then to an analytic versal deformation.
\bigskip

Before we prove the existence of a semiuniversal deformation of an  analytic algebra $R= K\{x\} /I$ with an isolated singularity, we need some more definitions. 

Recall that  $R$ is {\em regular} iff 
$\dim(R) = \dim_K \fm/\fm^2$, where $\dim(R)$ is the Krull dimension of $R$ and $\fm= \fm_R$ the maximal ideal of $R$.
Let $I$ be generated by $f_1,...,f_m$ and let $Jac(I)$ be
the Jacobian matrix $(\frac{\partial f_j}{\partial x_i})$. We denote by by $I_k(Jac(I))$ the ideal generated by the $k\times k$-minors of $Jac(I)$ (which is independent of the chosen generators $f_j$).
By the Jacobian criterion $R$ is regular iff $I_d(Jac(I)) = K\{x\}$ with $d=\dim(R)$, 
and this is equivalent to $R\cong K\{y_1,..,y_d\}$   by the implicit function theorem. 
Moreover, if $r$ is the minimal cardinality of a system of generators of $I$, then $R$ is regular iff $Jac(I)_{x=0}$ has rank $n-r$, and then $d=n-r$.
We define the {\em singular locus} of $R$ (or of $I$) as a subset of $\Spec R$ by
$$\Sing(R) := \{P \in \Spec R \mid R_P \text{ is not regular} \}.$$ 
We say that  $R$ (or $I$) has an {\em isolated singularity} (at 0) if the maximal ideal  $\fm$ is an isolated point of $\Sing(R)$ or if $\fm$ is a non-singular point.

\begin{Remark}{\em \label{rm.sing} The following description of $\Sing(R)$  follows easily from the Jacobian criterion. It shows in particular that $\Sing (R)$ is a closed subset of $\Spec R$.

(1) If  $R$ is pure $d$-dimensional (i.e. $\dim R/P =d$ for all minimal primes $P\in \Spec R$) then
$$\Sing (R) = V(I+ I_d(Jac(I))$$
with $V(J) = \{P\in \Spec R\mid P\supset J\}$.

(2) If $R$ is not pure dimensional we consider the minimal primes $P_1,\dots,P_r$ of $R$. Then $R/P_i$ is pure dimensional and we define the singular locus of $R$ as
\begin{align*}
\Sing (R) = \bigcup\nolimits_{i=1}^r \Sing (R/P_i) \cup  \bigcup\nolimits_{i\neq j}V(P_i)\cap V(P_j),
\end{align*}
The points in $\Spec R \setminus \Sing (R)$ are called {\em non-singular} or {\em smooth} points of $R$.}
\end{Remark}

Grauert's approximation theorem is by far the most difficult part of Grauert's original proof for the existence of a semiuniversal deformation of an arbitrary complex analytic germ with an isolated singularity. The same holds for analytic algebras over real valued fields in the following theorem.

\begin{Theorem}[Semiuniversal deformation over a real valued field]
\label{thm.sudef}
Assume that the analytic $K$-algebra $R=K\{x\} /I$, $K$ a real valued field, has an isolated singularity.  Then $R$ has a semiuniversal deformation.
\end{Theorem}
Instead of requiring that $R$ has an isolated singularity, it is sufficient to require that $\dim_K T^1_R < \infty$, where $T^1_R$ denotes the $K$-vector space of isomorphism classes of infinitesimal deformations of $R$, that is, deformations  over $T_\eps = K + \eps K$, $\eps^2 =0$. The support $T^1_R$ is contained in $\Sing(R)$, implying that $\dim_K T^1_R < \infty$ if $R$ has an isolated singularity. We refer to \cite [Section 2.1.4]{GLS25} or \cite[Section 3]{St03} for a discussion of $T^1_R$.
A different approach to proving the existence of a semiuniversal deformation of an isolated singularity using Banach-analytic techniques was given in \cite{Ha85}.

We mention that previously to Grauert, Schlessinger had proved the existence of a (in a certain sense formally) semiuniversal deformation of a formal analytic singularity $R=K[[x]] /I$, $K$ any field and $T^1_R$ finite dimensional, as well as Tjurina for a normal isolated singularity $R=\C\{x\} /I$. Grauert proved his theorem after a joint Oberseminar with Brieskorn in G\"ottingen in 1971 (the first author was participating as a student) in order to understand the results of Schlessinger and Tjurina. Finally, Grauert did not use their results, but developed his division and approximation theorem from scratch. In particular, Grauert did not use the existence of a formal solution (as one might expect in view of the standard applications of Artin's approximation theorems), but constructed a convergent approximation theorem in one go with estimates in every step (as we did  in Theorem \ref{grauertapprox}). 

\begin{proof}
The original proof of Grauert has been  presented in slightly different forms  by several authors. We refer to the book of Stevens \cite{St03} for a particular clear presentation and for further references. 
The step to derive Theorem \ref{thm.sudef} from Theorem  \ref{grauertapprox} is purely algebraic and much easier than the proof of the approximation theorem. It works in the same manner for $K\{x\}$ as for $\C\{x\}$, see \cite[Section 9]{St03}, so that we do not reproduce it here but refer to \cite{St03}. 
This proves  the existence of a convergent semiuniversal deformation as claimed in Theorem \ref{thm.sudef}.
\end{proof}

For general isolated singularities one cannot say anything about the semiuniversal deformation beyond its existence. For complete intersections however, in particular for hypersurfaces,  an explicit description can be given. This is well known for complex analytic and formal singularities. We are going to prove this for analytic complete intersections over arbitrary real valued fields $K$.
The analytic algebra $R=K\{x\}/I$, $x=(x_1,...,x_n)$, is called a {\it complete intersection}, if a minimal set of generators $f_1,...,f_k \in K\{x\}$ of $I$  satisfies  $k = n-\dim R$. We say also that $I$ is a complete intersection. We call the analytic algebra $R$ an ICIS if $R$ is a complete intersection with isolated singularity. We note that our proof, which uses Grauert's approximation theorem over $K$, is much shorter than the original proof over $\C$ given in \cite{KaS}.

\begin{Theorem} \label{thm.suICIS}
Let $R=K\{x\}/I$ be an ICIS, $K$ a real valued field, and $I$  minimally generated by $f_1,...,f_k \in K\{x\}$. Let $g_1,\dots,g_{\tau}\in R^k$,
$g_i=(g_i^1,\dots,g_i^k)$,  represent a basis
(respectively a system of generators) for the
finite dimensional $K$-vector space
$$ T^1_R  = K\{x\}^k\big/\textstyle\Bigl\langle\big(\frac{\partial f_i}{\partial x_j}\big)\cdot K\{x\}^n+\langle
f_1,\dots,f_k\rangle K\{x\}^k\Big\rangle $$
with $\big(\frac{\partial f_i}{\partial x_j}\big) : K\{x\}^n \to K\{x\}^k$  the Jacobian matrix of $(f_1,...,f_k)$.
With $t_1,\dots,t_\tau$  new variables
let $t=(t_1,\dots,t_\tau)$ and set
\begin{eqnarray*}
 F_1 & := & f_1+\sum_{j=1}^{\tau} t_jg_j^1,\\
[-0.6em]
\vdots & & \qquad \vdots\\
F_k  & := & f_k +\sum_{j=1}^{\tau} t_jg_j^k,\\
\kr & := & K\{x,t\} / \langle F_1, ... , F_k \rangle.
\end{eqnarray*}
Then $T=K\{t\} \to \kr \to R=\kr / \langle t_1,\dots,t_\tau \rangle$,
with $K\{t\} \to \kr$ induced by the inclusion $K\{t\} \to K\{x,t\}$ and 
$\kr \to R$ the canonical projection,
is a semiuniversal (respectively versal) deformation of
$R$.
\end{Theorem}

Note that $\big(\frac{\partial f_i}{\partial x_j}\big)\cdot R^n$ is the submodule of $R^k$ generated by the columns of $\big(\frac{\partial f_i}{\partial x_j}\big)$ and that
$\tau = \dim_K T^1_R,$
the {\it Tjurina number} of $R$, is finite since $R$ is an ICIS.

\begin{Corollary} \label{cor.ICIS}
Let $R=K\{x\}/\langle f \rangle$ be an isolated hypersurface singularity and $g_1,\dots,g_{\tau}\in K\{x\}$ representatives of
a $K$-basis (respectively a system of generators) of the
 Tjurina algebra
$$ T^1_R= K\{x\}/\textstyle\big\langle f,\frac{\partial
f}{\partial x_1},\dots,\frac{\partial f}{\partial
x_n}\big\rangle.$$
Setting $F= f+\sum_{j=1}^{\tau} t_jg_j$
and
$\kr =  K\{x,t\} / \langle F \rangle$, then 
$K\{t\} \to \kr \to R$
is a semiuniversal (respectively versal) deformation of
$R$.
\end{Corollary}
\begin{proof}[Proof of Theorem \ref{thm.suICIS}]
Let $F=(F_1,\ldots,F_k)$. It is enough to prove the theorem for deformations over the base space $K\{s\}$,  $s=(s_1,\ldots,s_q)$, since a deformation over $K\{s\}/I$ induces a deformation over $K\{s\}$ (this follows, for example, from \cite [Proposition 217]{GLS25}).
Let $G(x,s)=(G_1(x,s),\ldots,G_k(x,s))\in K\{x,s\}^k $ be a deformation of $f=(f_1,\ldots,f_k)$, i.e.
$$G(x,0)=M_0f(x)$$
for an invertible $k\times k$-matrix $M_0\in K\{x\}^{k^2}$. Considering $\tilde G:=M_0^{-1}G$ instead of $G$ we may assume that $M_0$ is the $k\times k$ unit matrix $\Id_k$.\\
 We have to prove that there exists $\phi(s)\in \langle s\rangle K\{s\}$, an automorphism  $\Phi$ of $K\{x,s\}$, $\Phi = (\Phi_1,...,\Phi_n)$, $\Phi_i = \Phi(x_i)$,
with $\Phi_i(x,0) = x_i$,
$\Phi(s_j) = s_j$, and a $k\times k$-matrix $M(x,s)$ with $M(x,0)=\Id_k$,
such that (see e.g. \cite[Proposition 4.3]{GPh19})
\begin{equation}\label{p64}
\Phi(G)=G(\Phi(x,s),s) = MF(x,\phi(s)).
\end{equation}
A solution of order $e$ (in $s$) of (\ref{p64}) is a triple $(\Phi,\phi,M) \in K\{x\}[s]^n \times K[s]^{\tau}  \times K[s]^{k^2}$  such that 
(\ref{p64}) holds mod $\langle s \rangle^{e+1}= \langle s \rangle^{e+1}K\{x,s\}^k$.

Since $\Phi(x,0) =x$, $\phi(0)= 0$, $M(x,0)=\Id_k$  and $G(x,0)=f(x)=F(x,0)$ we obviously have a solution of (\ref{p64}) of order $0$. 
To apply Theorem \ref{grauertapprox}, we have to show that every solution 
$(\Phi,\phi,M)$ of $(\ref{p64})$ of order $e$ in $s$ can be extended to a solution
$(\Phi',\phi',M')$ of $(\ref{p64})$ of order $e+1$. 
Having the solution $(\Phi,\phi,M)$ of order $e$, the difference
$G(\Phi(x,s),s) - MF(x,\phi(s))$ is of order $\ge e+1$ in $s$. The homogeneous part of degree $e+1$ in $s$ of
$G(\Phi(x,s),s) - MF(x,\phi(s))$ can thus be written as
\[
 G(\Phi(x,s),s) - MF(x,\phi(s)) = \sum_{| \nu | = e+1}s^{\nu}h_{\nu}(x) \mod \langle s\rangle^{e+2}
\]
with $h_{\nu} \in K\{x\}^k$. 
By assumption $g_1,...,g_{\tau}$ generate\\ 
$K\{x\}^k\big/\textstyle\Bigl\langle\big(\frac{\partial f_i}{\partial x_j}\big)\cdot K\{x\}^n+\langle
f_1,\dots,f_k\rangle K\{x\}^k\Big\rangle$
and therefore we can write for all $\nu$ 
\[
h_{\nu} =  \sum_{j=1}^{\tau}b_{\nu j}g_j + \sum_{j=1}^nh_{\nu j}
\frac{\partial f}{\partial x_j} + L_\nu f
\]
with $b_{j \nu } \in K$, $h_{\nu j} \in K\{ x\}$ and $L_\nu\in K\{x\}^{k^2}$ a $k\times k$-matrix..
We define 
\begin{align*}
\Phi'_j &= \Phi_j - \sum_{|\nu | = e+1}{}{h_{\nu j}}s^\nu \; \quad
j=1,\ldots,n\\
\phi'_j &= \phi_j + \sum_{|\nu | = e+1}b_{j \nu}s^\nu\;\quad j = 1, \ldots, p,\\
M' &= M +\sum_{|\nu | = e+1}s^\nu L_\nu.
\end{align*}
As $G(x,0)= f$ and $\Phi(x,0)=x$ we obtain 
$$\frac{\partial{(G(\Phi(x,s),s))}}{\partial{x_j}}\equiv \frac{\partial f}{\partial x_j}  \mod \langle s\rangle .$$
Using Taylor's formula  (see Remark \ref{rm.unf}) with $\Phi' = (\Phi'_1,...,\Phi'_n)$ and $\phi' = (\phi'_1,...,\phi'_p)$, we get
\begin{align*}
G(\Phi'(x,s),s)&=  G(\Phi(x,s),s)-\sum_{j=1}^n \frac{\partial G}{\partial x_j}(\Phi(x,s),s)(\sum_{|\nu | = e+1}{}{h_{\nu j}}s^\nu) \mod \langle s\rangle^{e+2}\\
&=  G(\Phi(x,s),s)-\sum_{|\nu | = e+1}(\sum_{j=1}^nh_{\nu j}\frac{\partial f}{\partial x_j} )s^\nu \mod \langle s\rangle^{e+2}.
\end{align*}
On the other hand, by Taylor and the definition of $F$,
\begin{align*}
F(x,\phi'(s))& = F(x,\phi(s)) +\sum_{j=1}^\tau(\frac{\partial F}{\partial t_j}(x,\phi(s))(\sum_{|\nu | = e+1}b_{j \nu}s^\nu) \mod \langle s\rangle^{e+2}\\
&=F(x,\phi(s))+\sum_{|\nu | = e+1}(\sum_{j=1}^\tau b_{j\nu}g_j)s^\nu \mod \langle s\rangle^{e+2}.
\end{align*}
This implies, since $M \mod \langle s\rangle =\Id_k$,
$$M'F(x,\phi'(s))
=(M+\sum_{|\nu | = e+1}s^\nu L_\nu)F(x,\phi'(s))$$
$$=MF(x,\phi(s))+\sum_{|\nu | = e+1}s^\nu L_\nu f +\sum_{|\nu | = e+1}(\sum_{j=1}^\tau b_{j\nu}g_j)s^\nu \mod \langle s\rangle^{e+2}.$$
We obtain
$$G(\Phi'(x),s)-M'F(x,\phi'(s))\equiv 0 \mod \langle s\rangle^{e+2}.$$
Now we can apply Theorem \ref{grauertapprox} to obtain an analytic solution of (\ref{def1}).
This proves that $F$ defines a complete deformation. We omit the proof of the versality of $F$, since it can be proved along the same lines, just with more complicated notations. 
To see that $F$ is semiuniversal if $g_1,\dots,g_{\tau}$ are a
basis of $T^1_R$ we have to show that the tangent map of
$\phi$ is uniquely determined. This follows as in the proof of \cite[Theorem 2.1.16]{GLS25}.
\end{proof}

{\bf Acknowledgement:} The first author would like to thank the Mathematical Institute of the University of Freiburg for its hospitality.


\bigskip

\noindent Gert-Martin Greuel\\
Department of Mathematics,
Rheinland-Pfälzische Technische Universität\\ Kaiserslautern-Landau (RPTU), Germany\\
Email address: greuel@mathematik.uni-kl.de
\medskip

\noindent Gerhard Pfister\\
Department of Mathematics,
Rheinland-Pfälzische Technische Universität\\ Kaiserslautern-Landau (RPTU), Germany\\
Email address: pfister@mathematik.uni-kl.de

\end{document}